\RequirePackage{fix-cm}
\documentclass[smallextended]{svjour3}       
\usepackage{lipsum}
\usepackage{amsfonts}
\usepackage{graphicx}
\usepackage{epstopdf}
\usepackage{algorithmic}
\ifpdf
  \DeclareGraphicsExtensions{.eps,.pdf,.png,.jpg}
\else
  \DeclareGraphicsExtensions{.eps}
\fi

\usepackage{amsopn}

\usepackage{fancyhdr}
\usepackage{amsmath}
\usepackage{graphicx}
\usepackage{amssymb}
\usepackage[stable]{footmisc}
\usepackage{algorithm,algorithmic}
\usepackage{multirow}

\newtheorem{assumption}{Assumption}

\newcommand{\ba}{\mathbf a}            
\newcommand{\bw}{\mathbf w}            

\usepackage{url}
\usepackage{color}
\smartqed  
\usepackage{graphicx}
\DeclareMathOperator*{\argmin}{arg\,min}

\DeclareMathOperator*{\Min}{minimize}
\begin{document}

\title{Fast L1-L2 Minimization via a Proximal Operator\thanks{This work was partially supported by the NSF grants DMS-1522786 and DMS-1621798.}
}


\author{Yifei Lou      \and
        Ming Yan 
}


\institute{ 
Y. Lou \at 
Department of Mathematical Sciences, University of Texas at Dallas\\
    \email{yifei.lou@utdallas.edu}
           \and					
M. Yan \at Department of Computational Mathematics, Science and Engineering (CMSE) and 
  	Department of Mathematics, Michigan State University\\ 
		\email{yanm@math.msu.edu}
}

\date{Received: date / Accepted: date}

\maketitle

\begin{abstract}
This paper aims to develop new and fast algorithms for recovering a sparse vector from a small number of measurements, which is a fundamental problem in the field of compressive sensing (CS). Currently, CS favors incoherent systems, in which any two measurements are as little correlated as possible. In reality, however, many problems are coherent, and conventional methods such as $L_1$ minimization do not work well. Recently, the difference of the $L_1$ and $L_2$ norms, denoted as $L_1$-$L_2$, is shown to have superior performance  over the classic $L_1$ method, but it is computationally expensive. 
We derive an analytical solution for the proximal operator of the $L_1$-$L_2$ metric, and it makes some fast $L_1$ solvers such as forward-backward splitting (FBS) and alternating direction method of multipliers (ADMM) applicable for $L_1$-$L_2$.
We describe in details how to incorporate the proximal operator into FBS and ADMM and show that the resulting algorithms are convergent under mild conditions. Both algorithms are shown to be much more efficient than the original implementation of $L_1$-$L_2$ based on a difference-of-convex approach in the numerical experiments.
\end{abstract}


\section{Introduction}

Recent developments in science and technology have caused a revolution in data processing, as large datasets are becoming increasingly available and important. To meet the need in ``big data'' era, the field of compressive sensing (CS)~\cite{donoho06,candesRT06} is rapidly blooming.
The process of CS consists of \textit{encoding} and \textit{decoding}. The process of encoding involves taking a set of (linear) measurements, $b = Ax$, where $A$ is a matrix of size $M\times N$. If $M<N$, we say the signal $x\in\mathbb{R}^N$ can be compressed.  The process of decoding is to recover $x$ from $b$ with an additional assumption that $x$ is sparse. It can be expressed as an optimization problem,
\begin{equation}\label{eq:CS}
\Min_x~\|x\|_0 \quad \mbox{subject to} \quad Ax = b,
\end{equation}
with $\|\cdot\|_0$ being the $L_0$ ``norm''. Since $L_0$ counts the number of non-zero elements, minimizing the $L_0$ ``norm'' is equivalent to finding the sparsest solution.

One of the biggest obstacles in CS is solving the decoding problem, eq.~(\ref{eq:CS}), as $L_0$ minimization is NP-hard ~\cite{natarajan95}.  A popular approach is to replace $L_0$ by a convex norm $L_1$, which often gives a satisfactory sparse solution. This $L_1$ heuristic has been applied in many different fields such as geology and geophysics~\cite{santosaS86}, spectroscopy~\cite{mammone83}, and ultrasound imaging~\cite{papoulisC79}. A revolutionary breakthrough in CS was the derivation of the restricted isometry property (RIP)~\cite{candesRT06}, which gives a sufficient condition of  $L_1$ minimization to recover the sparse solution exactly. It was proved in~\cite{candesRT06} that random matrices satisfy the RIP with high probabilities, which makes RIP seemingly applicable. 
However,  it is NP-hard to verify the RIP for a given matrix. 
A deterministic result in~\cite{donohoE03,gribonval2003sparse} says that exact sparse recovery using $L_1$ minimization is possible if
\begin{equation}\label{ineq:sparseVScoherence}
\textstyle \|x\|_0<1/2 +1/(2\mu),
\end{equation}
where $\mu$ is the  mutual coherence of a matrix $A$,  defined as
\begin{equation*}
\mu(A) = \max_{i\neq j} \dfrac{|\ba_i^\top\ba_j|}{\|\ba_i\|_2\|\ba_j\|_2}, \quad \mbox{with} \ A=[\ba_1, \cdots, \ba_N].
\end{equation*}
The inequality (\ref{ineq:sparseVScoherence}) suggests that $L_1$ may not perform well
for highly coherent matrices. When the matrix is highly coherent, we have $\mu\sim 1$, then the sufficient condition $\|x\|_0\leq 1$ means that $x$ has at most one non-zero element.

Recently, there has been an increase in applying nonconvex metrics as alternative approaches to $L_1$.
In particular, the nonconvex metric $L_p$ for $p\in(0,1)$
in~\cite{chartrand07,chartrandY08,krishnanF09,Xu2012} can be regarded as a continuation strategy to approximate $L_0$ as $p\rightarrow 0$. 
The optimization strategies include iterative reweighting~\cite{chartrand07,chartrandY08,laiXY13} and half thresholding~\cite{woodworth2015compressed,wuSL15,Xu2012}. The scale-invariant $L_1$, formulated as the ratio of $L_1$ and $L_2$, was discussed in~\cite{esserLX13,repetti2015euclid}.  Other nonconvex $L_1$ variants include transformed $L_1$~\cite{zhangX14}, sorted $L_1$~\cite{huangSY15}, and capped $L_1$~\cite{louYX15}.
It is demonstrated in a series of papers~\cite{louOX15,louYHX14,yinLHX14} that the difference of the $L_1$ and $L_2$ norms, denoted as $L_1$-$L_2$, outperforms $L_1$ and $L_p$ in terms of promoting sparsity when sensing matrix $A$ is highly coherent. Theoretically, a RIP-type sufficient condition is given in~\cite{yinLHX14} to guarantee that $L_1$-$L_2$ can exactly recover a sparse vector.

In this paper, we generalize the $L_1$-$L_2$ formalism by considering the $L_1-\alpha L_2$ metric for $\alpha\geq 0$. Define
\begin{align*}
r_\alpha(x)=\|x\|_1-\alpha\|x\|_2.
\end{align*}
We  consider an unconstrained minimization problem to allow the presence of noise in the data, i.e.,
\begin{align}\label{eq:main_pro}
\Min_x~E(x) \equiv r_\alpha(x)+l(x),
\end{align}
where $l(x)$ has a Lipschitz continuous gradient with Lipschitz constant $L$. 
Computationally, it is natural to apply difference-of-convex algorithm (DCA)~\cite{TA98} to minimize the $L_1$-$L_2$ functional. 
The DCA decomposes the objective function as the difference of two convex functions, i.e., $E(x) = G(x)-H(x)$, where
\begin{equation*}
\left\{\begin{array}{l}
G(x) = l(x) + \|x\|_1,\\
H(x) = \alpha\|x\|_2 .
\end{array}\right.\end{equation*}
Then, giving an initial $x^0\neq\mathbf{0}$, we obtain the next iteration by linearing $H(x)$ at the current iteration, i.e.,
\begin{align}
x^{n+1} \in&\textstyle \argmin_{x}~l(x)  + \|x\|_1-\|x^n\|_2-\left\langle x-x^n,\alpha\frac{x^n}{\|x^n\|_2}\right\rangle\nonumber\\
=&\textstyle \argmin_{x}~l(x)  + \|x\|_1-\alpha\left\langle x,\frac{x^n}{\|x^n\|_2}\right\rangle. \label{iter}
\end{align}
It is an $L_1$ minimization problem, which may not have analytical solutions  and usually requires to apply iterative algorithms.
It was proven in~\cite{yinLHX14} that the iterating sequence~\eqref{iter} converges to a stationary point of the unconstrained problem~\eqref{eq:main_pro}. Note that the DCA for $L_1$-$L_2$ is equivalent to alternating mininization for the following optimization problem:
\begin{equation*}
\Min_{x,q\in \mathbb R^N, \|q\|_2\leq 1} l(x)  + \|x\|_1+\alpha\langle x,q\rangle,
\end{equation*}
because $q=-{x\over\|x\|_2}$ for any fixed $x$.
Since DCA for $L_1$-$L_2$ amounts to solving an  $L_1$ minimization problem iteratively as a subproblem,  it is much slower than $L_1$ minimization. This motivates fast approaches proposed in this work.

We propose fast approaches for minimizing \eqref{eq:main_pro}, which are approximately of the same computational complexity as $L_1$. The main idea is based on a proximal operator corresponding to $L_1$-$\alpha L_2$. We then consider two numerical algorithms: forward-backward splitting (FBS) and alternating direction method of multipliers (ADMM), both of which are proven to be convergent under mild conditions.
The contributions of this paper are:
\begin{itemize}
\item We derive analytical solutions for the proximal mapping of $r_\alpha(x)$ in Lemma~\ref{lemma:prox}.
\item We propose a fast algorithm---FBS with this proximal mapping---and show its convergence in Theorem~\ref{thm:fbs}. Then, we analyze the properties of fixed points of FBS and show that FBS iterations are not trapped at stationary points near~$\mathbf{0}$ if the number of non-zeros is greater than one. It explains that FBS tends to converge to sparser stationary points when the $L_2$ norm of the stationary point is relatively small; see Lemma~\ref{lemma:nece} and Example~\ref{example:FBS}.
\item We propose another fast algorithm based on ADMM and show its convergence in Theorem~\ref{thm:admm}. This theorem applies to a general problem--minimizing the sum of two (possibly nonconvex) functions where one function has a Lipschitz continuous gradient and the other has an analytical proximal mapping or the mapping can be computed easily.
\end{itemize}

The rest of the paper is organized as follows. We detail the proximal operator in Section~\ref{sect:main}. The  numerical algorithms (FBS  and ADMM) are described in Section~\ref{sect:fbs} and Section~\ref{sect:admm}, respectively, each with convergence analysis. In  Section~\ref{sect:experiments}, we numerically compare the proposed methods with the DCA on different types of sensing matrices. During experiments, we observe a need to apply a continuation strategy of $\alpha$ to improve sparse recovery results. Finally, Section~\ref{sect:conclusion} concludes the paper.

\section{Proximal operator}\label{sect:main}
In this section, we present a closed-form solution of the proximal operator for $L_1$-$\alpha L_2$,  defined as follows,
\begin{equation}\label{eq:L12_uncon}
\textstyle \mbox{\bf prox}_{\lambda r_\alpha}(y)=\argmin_{x}~\|x\|_1-\alpha\|x\|_2 + \frac{1}{2\lambda} \|x-y\|_2^2,
\end{equation}
for a positive parameter $\lambda>0$. Proximal operator is particularly useful in convex optimization~\cite{rockafellar2015convex}. For example, the proximal operator for $L_1$ is called \textit{soft shrinkage}, defined as
\begin{equation*}
\mathcal S_1(y,\lambda) = \left\{
\begin{array}{ll}
y-\lambda, & \mbox{if } y>\lambda,\\
0, &\mbox{if } |y|\leq \lambda,\\
y+\lambda, &\mbox{if } y<-\lambda.
\end{array}
\right.
\end{equation*}
The soft shrinkage operator is a key for rendering many efficient $L_1$ algorithms. By replacing the soft shrinkage with $\mbox{\bf prox}_{\lambda r_\alpha}$, most fast $L_1$ solvers such as FBS and ADMM are applicable for $L_1$-$\alpha L_2$, which will be detailed in Sections~\ref{sect:fbs} and~\ref{sect:admm}.
The closed-form solution of  $\mbox{\bf prox}_{\lambda r_\alpha}$ is characterized in Lemma~\ref{lemma:prox}, while Lemma~\ref{lemma:assumption_not_satisfy} gives an important inequality to prove the convergence of FBS and ADMM when combined with the proximal operator.

\begin{lemma}\label{lemma:prox}
Given $y\in\mathbb R^N$, $\lambda>0$, and $\alpha\geq 0$, we have the following statements about the optimal solution $x^*$ to the optimization problem in~\eqref{eq:L12_uncon}:
\begin{enumerate}
\item[1)] When $\|y\|_\infty >\lambda$, $x^*= z(\|z\|_2+\alpha\lambda)/\|z\|_2$ for $z = \mathcal S_1(y,\lambda)$.
\item[2)] When $\|y\|_\infty =\lambda$, $x^*$ is an optimal solution if and only if it satisfies $x^*_i=0$ if $|y_i|<\lambda$, $\|x^*\|_2=\alpha\lambda$, and $x^*_iy_i\geq 0$ for all $i$. When there are more than one components having the maximum absolute value $\lambda$, the optimal solution is not unique; in fact, there are infinite many optimal solutions.
\item[3)] When $(1-\alpha)\lambda<\|y\|_\infty <\lambda$, $x^*$ is an optimal solution if and only if it is a 1-sparse vector satisfying $x^*_i=0$ if $|y_i|<\|y\|_\infty$, $\|x^*\|_2=\|y\|_\infty+(\alpha-1)\lambda$, and $x^*_iy_i\geq 0$ for all $i$. The number of optimal solutions is the same as the number of components having the maximum absolute value $\|y\|_\infty$.
\item[4)] When $\|y\|_\infty\leq(1-\alpha)\lambda$, $x^*=0$.
\end{enumerate}
\end{lemma}

\begin{proof}
It is straightforward to obtain the following relations about the sign and order of the absolute values for the components in $x^*$, i.e.,
\begin{align*}
x^*_i\left\{\begin{array}{ll}\geq 0, &\mbox{ if }y_i>0,\\
\leq 0, &\mbox{ if }y_i<0,\end{array}\right.
\end{align*}
and 
\begin{align}\label{eqn:prox_order}
|x^*_i|\geq |x^*_j|\mbox{ if } |y_i|>|y_j|.
\end{align}
Otherwise, we can always change the sign of $x^*_i$ or swap the absolute values of $x^*_i$ and $x^*_j$ and obtain a smaller objective value.
Therefore, we can assume  without loss of generality that $y$ is a non-negative non-increasing vector, \emph{i.e.}, $y_1\geq y_2\geq \cdots\geq y_N\geq 0$.

Denote $F(x)=\|x\|_1-\alpha\|x\|_2 + \frac {1} {2\lambda} \|x-y\|_2^2$  and
the first-order optimality condition of minimizing $F(x)$  is expressed as
\begin{equation}\label{eq:opt4L12}
\left(1-\dfrac {\alpha\lambda} {\|x\|_2}\right)x = y-\lambda p \quad \mbox{for}\ x\neq0,
\end{equation}
where $p\in\partial\|x\|_1$ is a subgradient of the $L_1$ norm. When $x=0$, we have the first order optimality condition $\|y-\lambda p\|_2=\alpha \lambda$. Simple calculations show that for any $x\neq 0$ satisfying~\eqref{eq:opt4L12}, we have
\begin{align*}
F(x)=&\textstyle \|x\|_1-\alpha \|x\|_2+{1\over 2\lambda}\|x\|_2^2-\langle x, p+\left({1\over\lambda}-\frac {\alpha}{\|x\|_2}\right)x\rangle +{1\over 2\lambda}\|y\|_2^2\\
=&\textstyle -\alpha \|x\|_2+{1\over 2\lambda}\|x\|_2^2-\left({1\over\lambda}-{\alpha\over \|x\|_2}\right)\|x\|_2^2 +{1\over 2\lambda}\|y\|_2^2\\
=&\textstyle -\frac {1} {2\lambda} \|x\|_2^2+\frac {1} {2\lambda} \|y\|_2^2 <F(0).
\end{align*}
Therefore, we have to find the $x^*$ with the largest norm among all $x$ satisfying~\eqref{eq:opt4L12}.
Now we are ready to discuss the four items listed in order,
\begin{enumerate}
\item[1)] If $y_1 >\lambda$, then $y_1-\lambda p_1>0$. For the case of $x^*\neq0$, we have $x^*_1>0$ and $1-\frac {\alpha\lambda}{\|x^*\|_2}>0$. 
For any $i$ such that $y_i\leq \lambda$, we have $x_i=0$; otherwise for this $i$, the left-hand side (LHS) of~\eqref{eq:opt4L12} is positive, while the right-hand side (RHS) is nonpositive.
For any $i$ such that $y_i>\lambda$, we have that $p_i=1$. 
Therefore, $y-\lambda p=\mathcal{S}_1(y,\lambda)$. 
Let $z = \mathcal S_1(y,\lambda)$, and we have $x^*= z(\|z\|_2+\alpha\lambda)/\|z\|_2$. 
Therefore, $x^*\neq 0$ is the optimal solution.

\item[2)] If $y_1 = \lambda$, then $y_1-\lambda p_1\geq 0$. Let $j=\min\{i: y_i<\lambda\}$, and we have $x^*_i=0$ for $i\geq j$; otherwise for this $i$, RHS of~\eqref{eq:opt4L12} is negative, and hence $1-\frac {\alpha\lambda}{\|x^*\|_2}<0$. It implies that $x_1^*=0$ and $x^*$ is not a global optimal solution because of~\eqref{eqn:prox_order}.
For the case of $x^*\neq 0$, we have $1-\frac {\alpha\lambda}{\|x^*\|_2}=0$. 
Therefore, any optimal solution $x^*$ satisfy that $x^*_i=0$ for $i\geq j$, $\|x^*\|_2=\alpha\lambda$, and $x^*_iy_i\geq 0$ for all $i$. When there are multiple components of $y$ having the same absolute value $\lambda$, there  exist infinite many solutions.

\item[3)] Assume $(1-\alpha)\lambda<y_1 <\lambda$. Let $j=\min\{i: y_i<\|y\|_\infty\}$, and we have $x^*_i=0$ for $i\geq j$; otherwise for this $i$, RHS of~\eqref{eq:opt4L12} is negative, thus $1-\frac {\alpha\lambda}{\|x^*\|_2}<0$ 
and $y_1-\lambda p_1=\left(1-\frac {\alpha\lambda}{\|x^*\|_2}\right)x^*_1\leq \left(1-\frac {\alpha\lambda}{\|x^*\|_2}\right)x^*_i=y_i-\lambda p_i$, which is a contradiction to $y_1>y_i$.
For the case of $x^*\neq 0$, we have $1-\frac {\alpha\lambda}{\|x^*\|_2}<0$.
From~\eqref{eq:opt4L12}, we know that $\alpha\lambda-\|x^*\|_2 =\|y-\lambda p\|_2$. 
Finding $x^*$ with the largest norm is equivalent to finding $p\in \partial \|x^*\|_1$ such that $\|y-\lambda p\|_2$ is smallest and $x^*\neq 0$. So we choose $x^*$ to be a 1-sparse vector, and $\|x^*\|_2=\alpha\lambda-\|y-\lambda p\|_2=\alpha\lambda-(\lambda-y_1)=y_1-(1-\alpha)\lambda$.


\item[4)] Assume that $y_1\leq (1-\alpha)\lambda$. 
If there exist an $x^*\neq 0$, we have $\|y-\lambda p\|_2\geq |y_1-\lambda|\geq \alpha \lambda$, while~\eqref{eq:opt4L12} implies $\|y-\lambda p\|_2=\alpha\lambda-\|x^*\|_2 < \alpha\lambda$. 
Thus we can not find $x^*\neq 0$. However, we can find $p\in\partial\|0\|_1$ such that $\|y-\lambda p\|_2=\alpha\lambda$. Thus $x^*=0$ is the optimal solution.
\end{enumerate}
\qed
\end{proof}

\begin{remark}When $\alpha=0$, $r_\alpha$ reduces to the $L_1$ norm and the proximal operator  $\mbox{\bf prox}_{\lambda r_\alpha}$ is equivalent to the soft shrinkage $\mathcal{S}_1(y,\lambda)$. When $\alpha>1$, items 3) and 4) show that the optimal solution can not be $0$ for any $y$ and positive $\lambda$.
\end{remark}

\begin{remark}During the preparation of this manuscript, Liu and Pong also provided an analytic solution for the proximal operator for the cases $0\leq \alpha \leq 1$ using a different approach~\cite{liu2016further}. In Lemma~\ref{lemma:prox}, we provide all the solutions for the proximal operator for any $\alpha\geq 0$. 
\end{remark}

\begin{lemma}\label{lemma:assumption_not_satisfy}
Given $y\in \mathbb R^N$, $\lambda>0$, and $\alpha\geq 0$. Let $F(x)=(\|x\|_1-\alpha\|x\|_2) + \frac 1 {2\lambda}  \|x-y\|_2^2$ and $x^*\in\mbox{\bf prox}_{\lambda r_\alpha}(y)$. Then, we have  for any $x\in\mathbb R^N$, 
\begin{equation*}
\textstyle F(x^*)- F(x) \leq \min\left({\alpha\over 2\|x^*\|_2}-{1\over 2\lambda} ,0\right)\|x^*-x\|_2^2.
\end{equation*}
Here, we let $\alpha/0$ be 0 when $\alpha=0$ and $+\infty$ for $\alpha>0$.
\end{lemma}
\begin{proof} When $\|y\|_\infty>(1-\alpha)\lambda$, Lemma~\ref{lemma:prox} guarantees that $\mbox{\bf prox}_{\lambda r_\alpha}(y)\neq 0$, i.e., $x^*\neq 0$. 
The optimality condition of $x^*$ reads $p={1\over\lambda}y-\left({1\over\lambda}-\frac {\alpha} {\|x^*\|_2}\right)x^*\in\partial \|x^*\|_1,$ 
then we have
\begin{align*}
\textstyle F(x^*)-F(x)\leq &\textstyle \langle p, x^*-x\rangle +\alpha\|x\|_2-\alpha \|x^*\|_2 +{1\over 2\lambda} \|x^* -y\|_2^2  -{1\over 2\lambda} \|x -y\|_2^2\\
= &\textstyle \left\langle  {\alpha x^*\over \|x^*\|_2}+{y-x^*\over \lambda}, x^*-x\right\rangle +\alpha\|x\|_2-\alpha \|x^*\|_2 \\
 &\textstyle -{1\over 2\lambda}\|x^*-x\|_2^2+{1\over \lambda} \langle x^*-y,x^*-x\rangle\\
= &\textstyle -\left\langle  {\alpha x^*\over \|x^*\|_2}, x\right\rangle +\alpha\|x\|_2-{1\over 2\lambda} \|x^*-x\|_2^2 \\
= &\textstyle {\alpha\over \|x^*\|_2}\left(-\langle x^*, x\rangle +\|x\|_2\|x^*\|_2 \right)-{1\over 2\lambda} \|x^*-x\|_2^2 \\
\leq &\textstyle {\alpha\over \|x^*\|_2}\left(-\langle x^*, x\rangle +{1\over 2}\|x\|_2^2+{1\over 2}\|x^*\|_2^2 \right)-{1\over 2\lambda} \|x^*-x\|_2^2 \\
= &\textstyle \left({\alpha\over 2\|x^*\|_2}-{1\over 2\lambda}\right) \|x^*-x\|_2^2.
\end{align*}
Here, the first inequality comes from $p\in\partial \|x^*\|_1$, and the last inequality comes from the Cauchy-Schwartz inequality. 

When $\|y\|_\infty\leq(1-\alpha)\lambda$, Lemma 1 shows that $x^*=\mbox{\bf prox}_{\lambda r_\alpha}(y)= 0$ and $F(x^*)-F(x)\leq 0$. Furthermore, if $\alpha=0$, we have 
\begin{align*}
 F(x^*)-F(x)= &\textstyle{1\over2\lambda}\|y\|^2-\|x\|_1-{1\over2\lambda}\|x-y\|^2\\
\leq &\textstyle{1\over\lambda}\langle x,y\rangle -\|x\|_1-{1\over2\lambda}\|x\|^2 \leq -{1\over2\lambda}\|x\|^2,
\end{align*}
where the last inequality holds because $\|y\|_\infty\leq \lambda$.
\qed
\end{proof}

\section{Forward-Backward Splitting}\label{sect:fbs}
Each iteration of forward-backward splitting applies the gradient descent of $l(x)$ followed by a proximal operator. It can be expressed as follows:
\begin{align*}
x^{k+1}\in\mbox{\bf prox}_{\lambda r_\alpha}(x^k-\lambda \nabla l(x^k)),
\end{align*}
where $\lambda>0$ is the stepsize. To prove the convergence, we make the following assumptions, which are standard in compressive sensing and image processing. 
\begin{assumption}\label{assume:Lipschitz}
	$l(x)$ has a Lipschitz continuous gradient, i.e., there exists $L>0$ such that 
	\begin{align*}
	\|\nabla l(x)-\nabla l(y)\|_2\leq L\|x-y\|_2\quad \forall~x,y.
	\end{align*}
\end{assumption}

\begin{assumption}\label{assume:coercive}
The objective function $r_\alpha (x)+l(x)$ is coercive, i.e., $r_\alpha (x)+l(x)\rightarrow +\infty$ when $\|x\|_2\rightarrow +\infty$.
\end{assumption}

The next theorem establishes the convergence of the FBS algorithm based on these two assumptions together with appropriately chosen stepsizes.

\begin{theorem}\label{thm:fbs}
If Assumptions~\ref{assume:Lipschitz}-\ref{assume:coercive}  are satisfied and $\lambda <1/L$, then the objective value is decreasing and there exists a subsequence that converges to a stationary point. 
In addition, any limit point is a stationary point of $E(x)$ defined in~\eqref{eq:main_pro}.
\end{theorem}
\begin{proof} Simple calculations give that 
\begin{align}
	   &\textstyle r_\alpha(x^{k+1}) + l(x^{k+1}) + \left({1\over 2\lambda}-{L\over 2}\right)\|x^{k+1}-x^k\|_2^2\nonumber\\
\leq &\textstyle r_\alpha(x^{k+1}) + l(x^{k}) +\left\langle \nabla l(x^k),x^{k+1}-x^k\right\rangle +{L\over 2}\|x^{k+1}-x^k\|_2^2 \nonumber\\
     &\textstyle + \left({1\over 2\lambda}-{L\over 2}\right)\|x^{k+1}-x^k\|_2^2\nonumber\\
=    &\textstyle r_\alpha(x^{k+1}) + l(x^{k})  + {1\over 2\lambda}\|x^{k+1}-x^k+\lambda \nabla l(x^k)\|_2^2-{1\over 2\lambda}\|\lambda \nabla l(x^k)\|_2^2\nonumber\\
\leq &\textstyle r_\alpha(x^k)  + {1\over 2\lambda}\|\lambda \nabla l(x^k)\|_2^2+ \min\left({\alpha\over 2\|x^{k+1}\|_2}-{1\over 2\lambda} ,0\right)\|x^{k+1}-x^k\|_2^2 \nonumber\\
     &\textstyle + l(x^{k})-{1\over 2\lambda}\|\lambda \nabla l(x^k)\|_2^2 \nonumber\\
\label{eqn:pd_ineq}
=    &\textstyle r_\alpha(x^k) + l(x^{k}) + \min\left({\alpha\over 2\|x^{k+1}\|_2}-{1\over 2\lambda} ,0\right)\|x^{k+1}-x^k\|_2^2.
\end{align}
The first inequality comes from Assumption~\ref{assume:Lipschitz}, and the second inequality comes from Lemma~\ref{lemma:assumption_not_satisfy} with $y$ replaced by $x^k-\lambda\nabla l(x^k)$ and $x$ replaced by $x^{k}$. Therefore, the function value $r_\alpha(x)+l(x)$ is decreasing; in fact, we have
\begin{eqnarray}\label{eq:fbs_suf_dec}
E(x^k)-E(x^{k+1})&\geq&\textstyle   \max\left({1\over \lambda}-{L\over 2}-{\alpha\over 2\|x^{k+1}\|_2} ,{1\over 2\lambda}-{L\over 2}\right)\|x^{k+1}-x^k\|_2^2\\
&\geq&\textstyle \left({1\over 2\lambda}-{L\over 2}\right)\|x^{k+1}-x^k\|_2^2.\nonumber
\end{eqnarray}

Due to the coerciveness of the objective function (Assumption~\ref{assume:coercive}), we have that the sequence $\{x^k\}_{k=1}^\infty$ is bounded. In addition, we have $\sum_{k=0}^{+\infty}\|x^{k+1}-x^k\|_2^2<+\infty$, which implies $x^{k+1}-x^k\rightarrow 0$. Therefore, there exists a convergent subsequence $x^{k_i}$. Let $x^{k_i}\rightarrow x^*$, then we have $x^{k_i+1}\rightarrow x^*$ and $x^*=\mbox{\bf prox}_{\lambda r_\alpha}(x^*-\lambda \nabla l(x^*))$, i.e., $x^*$ is a stationary point. \qed
\end{proof}

\begin{remark}
When $\alpha=0$, the algorithm is identical to the iterative soft thresholding algorithm (ISTA)~\cite{Beck2009}, and the stepsize can be chosen as $\lambda<2/L$  since~\eqref{eq:fbs_suf_dec} becomes
$$\textstyle r_\alpha(x^k) + l(x^{k}) - r_\alpha(x^{k+1}) + l(x^{k+1})\geq  \left({1\over \lambda}-{L\over 2} \right)\|x^{k+1}-x^k\|_2^2.$$ 
When $\alpha>0$, if we know a lower bound of $\|x^k\|_2$, we may choose a larger stepsize to speed up the convergence based on the inequality~\eqref{eqn:pd_ineq}.  
\end{remark}

\begin{remark}The result in Theorem~\ref{thm:fbs} holds for any regularization $r(x)$, and the proof follows from replacing $\min\left({\alpha\over 2\|x^{k+1}\|_2}-{1\over 2\lambda} ,0\right)$ in~\eqref{eqn:pd_ineq} by $0$~\cite[Proposition 2.1]{bredies_minimization_2015}.
\end{remark}

Since the  main problem~\eqref{eq:main_pro} is nonconvex, there exist many stationary points. We are interested in those stationary points that are also fixed points of the FBS operator because a global solution is a fixed point of the operator and FBS converges to a fixed point. In fact, we have the following property for global minimizers to be fixed points of the FBS algorithm for all parameters $\lambda<1/L$.

\begin{lemma}\label{lemma:nece}
[Necessary conditions for global minimizers] Each global minimizer $x^*$ of~\eqref{eq:main_pro} satisfies:
\begin{itemize}
\item[1)] $x^{*}\in\mbox{\bf prox}_{\lambda r_\alpha}(x^*-\lambda \nabla l(x^*))$ for all positive $\lambda <1/L$.
\item[2)] If $x^*=0$, then we have $\|\nabla l(0)\|_\infty \leq 1-\alpha$. In addition, we have $\nabla l(0)=0$ for $\alpha=1$ and $x^*=0$ does not exist for $\alpha >1$.
\item[3)] If $\|x^*\|_2\geq \alpha/L$,let $\Lambda=\{i,x^*_i\neq 0\}$. Then $x^*_\Lambda$ is in the same direction of $\nabla_\Lambda l(x^*)+\mbox{sign}(x^*_\Lambda)$ and $\|\nabla_\Lambda l(x^*)+\mbox{sign}(x^*_\Lambda)\|_2 = \alpha$. 
\item[4)] If $\|x^*\|_2< \alpha/L$ and $x^*\neq0$, then $x^*$ is 1-sparse, i.e., the number of nonzero components is 1. In addition, we have $\nabla_i l(x^*)=(\alpha -1)\mbox{sign}(x^*_i)$ for $x^*_i\neq 0$ and $|\nabla_i l(x^*)|\leq \min\{0,1-\alpha+\|x^*\|_\infty L\}$ for $x^*_i= 0$. 
\end{itemize}
\end{lemma}

\begin{proof} 
Item 1) follows from~\eqref{eq:fbs_suf_dec} by replacing $x^k$ with $x^*$. The function value can not decrease because $x^*$ is a global minimizer. Thus $x^{k+1}=x^*$, and $x^*$ is a fixed point of the forward-backward operator. 
Let $x^*=0$, then item 1) and Lemma~\ref{lemma:prox} together give us item 2). 

For items 3) and 4), we denote $y(\lambda)=x^*-\lambda\nabla l(x^*)$ and have $\|y(\lambda)\|_\infty >\lambda$ for small positive $\lambda$ because $x^*\neq 0$. 
If $\|x^*\|_2\geq \alpha/L$, then from Lemma~\ref{lemma:prox}, we have that $\|x^*\|_2\geq \alpha\lambda$ and $\|y(\lambda)\|\geq \lambda$ for all $\lambda <1/L$.
Therefore, we have $x^*=S_1(y,\lambda)(\|S_1(y,\lambda)\|_2+\alpha \lambda )/\|S_1(y,\lambda)\|_2$ for all $\lambda<1/L$ from Lemma~\ref{lemma:prox}. 
$S_1(y,\lambda)$ is in the same direction of $x^*$, and thus $x^*_\Lambda$ is in the same direction of $\nabla_\Lambda l(x^*)+\mbox{sign}(x^*_\Lambda)$. 
In addition, $\|\nabla_\Lambda l(x^*)+\mbox{sign}(x^*_\Lambda)\|_2 = \alpha$. 
If $\|x^*\|_2< \alpha/L$, then from Lemma~\ref{lemma:prox}, we have that $x^*$ is 1-sparse. We also have $\nabla_i l(x^*)=(\alpha -1)\mbox{sign}(x^*_i)$ for $x^*_i\neq 0$, which is from Item 3). 
For $x^*_i=0$, we have $|\lambda \nabla_i l(x^*)|\leq \left|\|x^*\|_\infty-\lambda (\alpha-1)\right|$ for all $\lambda<1/L$. Thus $| \nabla_il(x^*)|\leq \left|\|x^*\|_\infty/\lambda- (\alpha-1)\right|$. When $\alpha<1$, we have $| \nabla_i l(x^*)|\leq 1-\alpha +\|x^*\|_\infty L$. When $\alpha>1$, if $1-\alpha+\|x^*\|_\infty L<0$, then we can find $\bar\lambda<1/L$ such that $1-\alpha+\|x^*\|_\infty/\bar\lambda=0$ and $|\nabla_i l(x^*) |\leq 0$, otherwise, we have $| \nabla_i l(x^*)|\leq 1-\alpha +\|x^*\|_\infty L$. \qed
\end{proof}


The following example shows that  FBS tends to select a sparser solution, i.e., the fixed points of the forward-backward operator may be sparser than other stationary points.

\begin{example}\label{example:FBS}
Let $N=3$ and the objective function be 
\begin{align*}
\textstyle\|x\|_1-\|x\|_2+ {1\over 2}\left(x_1+x_2-1.2+1/\sqrt{2}\right)^2 +{1\over 2}\left(x_2+x_3-1.2+1/\sqrt{2}\right)^2.
\end{align*}
We can verify that $\left(0,1.2-{1/\sqrt{2}},0\right)$ is a global minimizer. In addition, we get $(0.2,0,0.2)$, $\left(1.2-{1/\sqrt{2}},0,0\right)$, $\left(0,0,1.2-{1/\sqrt{2}}\right)$, and $(4/5-2/9-\sqrt{2}/3,2/5-1/9-\sqrt{2}/6,4/5-2/9-\sqrt{2}/3)$ are stationary points. Let $x^0=(0,0,0)$, we have that $x^*=\left(0,1.2-{1/\sqrt{2}},0\right)$. If we let $x^0=(0.2,0,0.2)$, we will have that $x^*=\left(1.2-{1/\sqrt{2}},0,0\right)$ (or $\left(0,0,1.2-{1/\sqrt{2}}\right)$), for stepsize $\lambda > 0.2\sqrt{2}\approx 0.2828$. Similarly,  if we let $x^0=(4/5-2/9-\sqrt{2}/3,2/5-1/9-\sqrt{2}/6,4/5-2/9-\sqrt{2}/3)$, we will have that $x^*=\left(0,1.2-{1/\sqrt{2}},0\right)$ for $\lambda>6/5-1/3-\sqrt{2}/2\approx 0.1596$. For both stationary points that are not 1-sparse, we can verify that their $L_2$ norms are less than $1/L=1/3$. Therefore, Lemma~\ref{lemma:nece} shows that they are not fixed points of FBS for all $\lambda<1/L$ and hence they are not global solutions. \end{example}

We further consider an accelerated proximal gradient method~\cite{li2015accelerated} to speed up the convergence of FBS. In particular, the algorithm goes as follows,
	\begin{subequations}
		\begin{align}
	&	y^k =  x^k + \frac{t^{k-1}}{t^k} (z^k-x^k) + \frac {t^{k-1}-1}{t^k} (x^k-x^{k-1}),\\
	&	z^{k+1} \in  \mbox{\bf prox}_{\lambda r_\alpha}(y^k-\lambda \nabla l(y^k)),\\
	&	v^{k+1} \in   \mbox{\bf prox}_{\lambda r_\alpha}(x^k-\lambda \nabla l(x^k)),\\
	&	t^{k+1} = \frac{\sqrt{4(t^k)^2+1}+1}{2},\\
	&	x^{k+1} =  
		\left\{\begin{array}{ll}
		z^{k+1}, & \mbox{if } E(z^{k+1})<E(v^{k+1}),\\
		v^{k+1}, & \mbox{otherwise.}
		\end{array}
		\right.
		\end{align}
	\end{subequations}
It was shown in~\cite{li2015accelerated} that the algorithm converges to a critical point if $\lambda<1/L$. We call this algorithm FBS throughout the numerical section. 

\section{Alternating Direction Method of Multipliers}\label{sect:admm}
In this section, we consider a general regularization $r(x)$ with an assumption that it is coercive; it includes $r_\alpha(x)$ as a special case. We  apply the ADMM to solve the unconstrained problem~\eqref{eq:main_pro}. In order to do this, we introduce an auxiliary  variable $y$ such that~\eqref{eq:main_pro} is equivalent to the following constrained minimization problem:
\begin{align}\label{eq:ADMM_pro}
\Min_{x,y}~r(x)+l(y) \  \mbox{ subject to }\ x=y.
\end{align} 
Then the augmented Lagrangian is 
\begin{align*}
L_\delta(x,y,u) = r(x)+l(y)+\delta\langle u,x-y\rangle +{\delta\over2}\|x-y\|_2^2,
\end{align*}
and the ADMM iteration is:
\begin{subequations}
\begin{align}
x^{k+1} \in & \argmin_x~L_\delta(x,y^k,u^k)=\argmin_x~r(x)+{\delta\over 2}\|x-y^k+u^k\|_2^2,\label{eq:ADMM_x}\\
y^{k+1} = & \argmin_y~L_\delta(x^{k+1},y,u^k)=\argmin_y~l(y)+{\delta\over 2}\|x^{k+1}-y+u^k\|_2^2,\label{eq:ADMM_y}\\
u^{k+1} = & u^k+x^{k+1}-y^{k+1}.\label{eq:ADMM_u}
\end{align}
\end{subequations}
Note that the optimality condition of~\eqref{eq:ADMM_y} guarantees that $0= \nabla l(y^{k+1})+\delta (y^{k+1}-x^{k+1}-u^k)$ and $\delta u^{k+1}=\nabla l(y^{k+1})$.

\begin{lemma}\label{lemma:ADMM} 
Let $(x^k,y^k,u^k)$ be the sequence generated by ADMM. We have the following statements:
\begin{itemize}
\item[1)] If $l(x)$ satisfies Assumption~\ref{assume:Lipschitz}, then we have
\begin{align}
\textstyle L_\delta(x^{k+1},y^{k+1},u^{k+1}) - L_\delta(x^{k+1},y^k,u^k)
\leq\left({3L\over 2} +{L^2\over \delta} -{\delta\over2}\right)\|y^{k+1}-y^k\|_2^2, \label{eq:suf_dec}
\end{align}
and, in addition, if $l(x)$ is convex,
\begin{align}
\label{eq:suf_dec_convex}
\textstyle L_\delta(x^{k+1},y^{k+1},u^{k+1}) - L_\delta(x^{k+1},y^k,u^k)
\leq \left({L^2\over \delta}-{\delta\over 2}\right)\|y^{k+1}-y^k\|_2^2.
\end{align}
\item[2)] If $l(x)$ satisfies Assumption~\ref{assume:Lipschitz}, then there exists $p\in\partial_x L_\delta(x^{k+1},y^{k+1},u^{k+1})$, where $\partial_x L_\delta$ is the set of general subgradients of $L$ with respect to $x$ for fixed $y$ and $u$~\cite[Deﬁnition 8.3]{rockafellar_variational_2009}
such that 
\begin{align}\label{eq:sub_diff_bound}
&\|p\|_2+\|\nabla_y L_\delta(x^{k+1},y^{k+1},u^{k+1})\|_2+\|\nabla_u L_\delta(x^{k+1},y^{k+1},u^{k+1})\|_2\\
\leq& (3L+\delta)\|y^{k+1}-y^k\|_2.\nonumber
\end{align}
\end{itemize}
\end{lemma}
\begin{proof}
1): From~\eqref{eq:ADMM_x}, we have 
\begin{align}
L_\delta(x^{k+1},y^k,u^k) - L_\delta(x^k,y^k,u^k)\leq 0.\label{eq:suf_dec_a}
\end{align}
From~\eqref{eq:ADMM_y} and~\eqref{eq:ADMM_u}, we derive
\begin{align}\label{for:equ1}
&L_\delta(x^{k+1},y^{k+1},u^{k+1}) - L_\delta(x^{k+1},y^k,u^k) \nonumber\\
=&l(y^{k+1})+\delta\langle u^{k+1},x^{k+1}-y^{k+1}\rangle +{\delta\over2}\|x^{k+1}-y^{k+1}\|_2^2\nonumber\\
 &-l(y^{k})-\delta\langle u^{k},x^{k+1}-y^{k}\rangle -{\delta\over2}\|x^{k+1}-y^{k}\|_2^2\nonumber\\
=&l(y^{k+1})-l(y^{k})-\delta\langle u^k,y^{k+1}-y^{k}\rangle \nonumber\\
& +\delta\|u^{k+1}-u^k\|_2^2 -{\delta\over2}\|y^{k+1}-y^k\|_2^2-\delta\langle u^{k+1}-u^k,y^{k+1}-y^{k}\rangle.
\end{align}
Assumption~\ref{assume:Lipschitz} gives us 
\begin{align*}
\textstyle l(y^{k+1})-l(y^{k})-\delta\langle u^k,y^{k+1}-y^{k}\rangle  \leq {L\over 2}\|y^{k+1}-y^k\|_2^2,
\end{align*}
and, by Young's inequality, we have 
\begin{align*}
-\delta\langle u^{k+1}-u^k,y^{k+1}-y^{k}\rangle\leq {c \delta}\|u^{k+1}-u^k\|_2^2 +{\delta\over 4c}\|y^{k+1}-y^k\|_2^2
\end{align*}
for any positive $c$ (we will decide $c$ later). Therefore we have 
\begin{align*}
&L_\delta(x^{k+1},y^{k+1},u^{k+1}) - L_\delta(x^{k+1},y^k,u^k) \\
\leq & {L\over 2}\|y^{k+1}-y^k\|_2^2 +{(1+ c) \delta}\|u^{k+1}-u^k\|_2^2 -\left({\delta\over2}-{\delta\over 4c}\right)\|y^{k+1}-y^k\|_2^2\\
\leq &{L\over 2}\|y^{k+1}-y^k\|_2^2 +{(1+ c)L^2\over \delta}\|y^{k+1}-y^k\|_2^2 -\left({\delta\over2}-{\delta\over 4c}\right)\|y^{k+1}-y^k\|_2^2.
\end{align*}
Let $c={\delta/(2L)}$, and we obtain:
\begin{align}
\textstyle L_\delta(x^{k+1},y^{k+1},u^{k+1}) - L_\delta(x^{k+1},y^k,u^k)
\leq\left({3L\over 2} +{L^2\over \delta} -{\delta\over2}\right)\|y^{k+1}-y^k\|_2^2.
\label{eq:suf_dec_b}\end{align}
Combining~\eqref{eq:suf_dec_a} and~\eqref{eq:suf_dec_b}, we get~\eqref{eq:suf_dec}.
If, in addition, $l(x)$ is convex, we have, from~\eqref{for:equ1}, that
\begin{align}\label{eq:suf_dec_convex_b}
&L_\delta(x^{k+1},y^{k+1},u^{k+1}) - L_\delta(x^{k+1},y^k,u^k)\nonumber\\
=&l(y^{k+1})-l(y^{k})-\delta\langle u^{k+1},y^{k+1}-y^{k}\rangle  +\delta\|u^{k+1}-u^k\|_2^2 -{\delta\over2}\|y^{k+1}-y^k\|_2^2\nonumber\\
\leq&\textstyle   \delta\|u^{k+1}-u^k\|_2^2 -{\delta\over2}\|y^{k+1}-y^k\|_2^2
\leq \left({L^2\over \delta}-{\delta\over 2}\right)\|y^{k+1}-y^k\|_2^2.
\end{align}
Thus~\eqref{eq:suf_dec_convex} is obtained by combining~\eqref{eq:suf_dec_a} and~\eqref{eq:suf_dec_convex_b}.

2) It follows from the optimality condition of~\eqref{eq:ADMM_x} that there exists $q\in \partial r(x^{k+1})$ such that 
\begin{align*}
q + \delta (x^{k+1}-y^k+u^k)=0.
\end{align*}
Let $p=q+\delta (u^{k+1}+x^{k+1}-y^{k+1})\in\partial_x L_\delta(x^{k+1},y^{k+1},u^{k+1})$, then we have 
\begin{align}\label{eq:sub_x}
&\|p\|_2=\|q+\delta (u^{k+1}+x^{k+1}-y^{k+1}) \|_2
=\|\delta (u^{k+1}-u^k+y^k-y^{k+1}) \|_2\nonumber\\
\leq &\delta\|u^{k+1}-u^k\|_2+\delta\|y^{k+1}-y^k\|_2
\leq  (L+\delta)\|y^{k+1}-y^k\|_2,
\end{align}
The optimality condition of~\eqref{eq:ADMM_y} and the update of $u$ in~\eqref{eq:ADMM_u} give that
\begin{align}
\|\nabla_y L_\delta(x^{k+1},y^{k+1},u^{k+1})\|=&\|\nabla l(y^{k+1})+\delta (-u^{k+1}+y^{k+1}-x^{k+1})\|\nonumber\\
    = &  \delta\|u^{k+1}-u^k\|\leq L\|y^{k+1}-y^k\|_2,\label{eq:sub_y}\\
\|\nabla_u L_\delta(x^{k+1},y^{k+1},u^{k+1})\|=& \|\delta (x^{k+1}-y^{k+1})\|=\delta\|u^{k+1}-u^k\| \nonumber\\
 \leq& L\|y^{k+1}-y^k\|_2. \label{eq:sub_u}
\end{align}
Thus~\eqref{eq:sub_diff_bound} is obtained by combining~\eqref{eq:sub_x},~\eqref{eq:sub_y}, and~\eqref{eq:sub_u}. \qed
\end{proof}

\begin{theorem}\label{thm:admm}
Let Assumptions~\ref{assume:Lipschitz} and~\ref{assume:coercive} be satisfied and $\delta > (3+\sqrt{17})L/2$ ($\delta > \sqrt{2}L$ if $l(x)$ is convex), then
\begin{itemize}
\item[1)] the sequence $(x^k,y^k,u^k)$ generated by ADMM is bounded and has at least one limit point.
\item[2)] $x^{k+1}-x^k\rightarrow 0$, $y^{k+1}-y^k\rightarrow 0$, and $u^{k+1}-u^k\rightarrow 0$.
\item[3)] each limit point $(x^*,y^*,u^*)$ is a stationary point of $L_\delta(x,y,u)$, and $x^*$ is a stationary point of $r(x)+l(x)$.
\end{itemize}
\end{theorem}
\begin{proof}
1) When $\delta>(3+\sqrt{17})L/2$, we have ${3L\over 2} +{L^2\over \delta} -{\delta\over2}<0$. In addition, for the case $l(x)$ being convex, we have ${L^2\over \delta}-{\delta\over 2}<0$ if $\delta>\sqrt{2}L$. There exists a positive constant $C_1$ that depends only on $L$ and $\delta$ such that 
\begin{align}
\label{eq:suf_dec_all}
L_\delta(x^{k+1},y^{k+1},u^{k+1}) - L_\delta(x^{k+1},y^k,u^k)
\leq -C_1\|y^{k+1}-y^k\|_2^2.
\end{align}
Next, we show that the augmented Lagrangian $L_\delta$ has a global lower bound during the iteration. From Assumption~\ref{assume:Lipschitz}, we have
\begin{align}
L_\delta (x^k,y^k,u^k) = & r (x^k)+l(y^k)+\delta\langle u^k,x^k-y^k\rangle +{\delta\over 2}\|x^k-y^k\|_2^2\nonumber\\
\geq & r(x^k) + l(x^k)+{\delta-L\over 2}\|x^k-y^k\|_2^2.\label{eq:low_bdd}
\end{align} 
Thus $L_\delta(x^k,y^k,u^k)$ has a global lower bound because of the coercivity of $r(x)+l(x)$ and $\delta>L$. It follows from~\eqref{eq:low_bdd} that $x^k, y^k, r(x^k)+l(x^k),$ and $\|x^k-y^k\|_2$ are all bounded. Therefore, $u^k$ is bounded because of Assumption~\ref{assume:Lipschitz}.

Due to the boundedness of $(x^k,y^k,u^k)$, there exists a convergent subsequence $(x^{k_i},y^{k_i},u^{k_i})$, i.e., $(x^{k_i},y^{k_i},u^{k_i})\rightarrow (x^*,y^*,u^*)$.

2) Since the sequence $L_\delta(x^k,y^k,u^k)$ is bounded below,~\eqref{eq:suf_dec_all} implies that $\sum_{k=1}^\infty\|y^{k+1}-y^k\|_2^2<\infty$ and $\|y^{k+1}-y^k\|_2^2\rightarrow 0$, i.e., $y^{k+1}-y^k\rightarrow0$. In addition, we have $u^{k+1}-u^{k}\rightarrow 0$ and $x^{k+1}-x^{k}\rightarrow 0$ due to Assumption~\ref{assume:Lipschitz} and ~\eqref{eq:ADMM_u} respectively.

3) Part 2 of Lemma~\ref{lemma:ADMM} and $y^{k+1}-y^k\rightarrow 0$ suggest that $(x^*,y^*,u^*)$ is a stationary point of $L_\delta(x,y,u)$. Since $(x^*,y^*,u^*)$ is a stationary point, we have $x^*=y^*$ from~\eqref{eq:sub_u}, then~\eqref{eq:sub_x} implies that $\delta u^*=\nabla l(y^*)$  and $0\in\partial_x r(x^*)+ \nabla l(x^*)$, i.e., $x^*$ is a stationary point of $r(x)+l(x)$.  \qed
\end{proof}

\begin{remark}
In~\cite{li2015global}, the authors show the convergence of the same ADMM algorithm when $l(y)=\|Ay-b\|_2^2$ and $\delta>\sqrt{2}L$, other choices of $l(x)$ are not considered in~\cite{li2015global}. 
The proof of Theorem~\ref{thm:admm} is inspired from~\cite{wang_global_2015}. 
Early versions of~\cite{wang_global_2015} on arXiv.org require that $r(x)$ is restricted prox-regular, while our $r(x)$ does not satisfy because it is positive homogeneous and nonconvex. However, we would like to mention that later versions of~\cite{wang_global_2015} after our paper cover our result.
\end{remark}

The following example shows that both FBS and ADMM may converge to a stationary point that is not a local minimizer.
\begin{example} Let $n=2$ and the objective function be 
\begin{align*}
\textstyle \|x\|_1-\|x\|_2+ {1\over 2}\|x_1+x_2-1\|_2^2.
\end{align*}
We can verify that $(1,0)$ and $(0,1)$ are two global minimizers with objective function value $0$. There is another stationary point $x^*=({1\over 2\sqrt{2}},{1\over 2\sqrt{2}})$ for this function. 
Assume that we assign the initial $x^0=(c^0,c^0)$ with $c^0> 0$, FBS generates $x^k=(c^k,c^k)$ where $c^{k+1}=(1-2\lambda )c^k+\lambda/\sqrt{2}$ for all $\lambda<1/L=1/2$. For ADMM, let $y^0=(d^0,d^0)$ and $u^0=(e^0,e^0)$ such that $e^0>0$ and $d^0>e^0+1/\delta$, then ADMM generates $x^k=(c^k,c^k)$, $y^k=(d^k,d^k)$, and $u^k=(e^k,e^k)$ with 
\begin{align*}
c^{k+1} & =\textstyle d^k-e^k-{1\over\delta}(1-{1\over \sqrt{2}}),\\
d^{k+1} & =\textstyle {\delta\over2+\delta}(c^{k+1}+e^k) +{1\over 2+\delta}={\delta\over2+\delta}d^k +{1\over 2+\delta}{1\over\sqrt{2}},\\
e^{k+1} & =\textstyle {2\over 2+\delta}(c^{k+1}+e^k)-{1\over 2+\delta}={2\over 2+\delta}d^k-{1\over 2+\delta}{1\over\sqrt{2}}-{1\over \delta}\left(1-{1\over\sqrt{2}}\right).
\end{align*}
\end{example}

\section{Numerical Experiments}\label{sect:experiments}
In this section, we compare our proposed algorithms with DCA on three types of matrices: random Gaussian, random partial DCT, and random over-sampled DCT matrices. 
Both random Gaussian and partial DCT matrices satisfy the RIP with high probabilities~\cite{candesRT06}. The size of these two types of matrices is $64\times 256$. 
Each entry of random Gaussian matrices follows the standard normal distribution, i.e., zero-mean with standard deviation of one, while we randomly select rows from the full DCT matrix to form partial DCT matrices.
The over-sampled DCT matrices are highly coherent, and they are derived from the problem of spectral estimation~\cite{fannjiangL12} in signal processing.
An over-sampled DCT matrix is defined as $A=[\ba_1, \cdots, \ba_N]\in\mathbb{R}^{M\times N}$ with
\begin{equation*}
\textstyle\ba_j = \frac 1 {\sqrt{N}} \cos\left(\dfrac {2\pi \bw j}{F}\right),\qquad j = 1, \cdots, N,
\end{equation*}
where $\bw$ is a random vector of length $M$ and $F$ is the parameter used to decide how coherent the matrix is. 
The larger $F$ is, the higher the coherence is.  We consider  two  over-sampled DCT matrices of size  $100\times 1500$ with $F=5$ and $F=20$. 
All the testing matrices are normalized to have unit (spectral) norm. 

As for the  (ground-truth) sparse vector, we generate the random index set and draw  non-zero elements following the standard normal distribution.
We compare the performance and efficiency of all algorithms in recovering the sparse vectors for both the noisy and noise-free cases. 
For the noisy case, we may also construct the noise such that the sparse vectors are stationary points. 	
The initial value for all the implementations  is chosen to be an approximated solution of the $L_1$ minimization, i.e.,
\[
\textstyle x^0 = \argmin_x~\gamma\|x\|_1 + \frac 1 2 \|Ax-b\|_2^2.
\]	
The approximated solution is obtained after 2N ADMM iterations. 
The stopping condition for the proposed FBS and ADMM is either $\|x^{k+1}-x^k\|_2/\|x^k\|_2<1e^{-8}$ or $k>10N$.


We examine the overall performance in terms of recovering exact sparse solutions for the noise-free case. 
In particular, we look at success rates with 100 random realizations. A trial is considered to be successful if the relative error of the reconstructed solution $x_r$ by an algorithm to the ground truth $x_g$ is less than .001, \emph{i.e.}, $\frac{\|x_r-x_g\|}{\|x_g\|}<.001$. For the noisy case, we compare the mean-square-error of the reconstructed solutions. All experiments are performed using Matlab 2016a on a desktop (Windows 7, 3.6GHz CPU, 24GB RAM). The Matlab source codes can be downloaded at \url{https://github.com/mingyan08/ProxL1-L2}.

\subsection{Constructed Stationary Points}

We construct the data term $b$ such that a given sparse vector $x^*$ is a stationary point of the unconstrained $L_1$-$L_2$ problem,
\begin{equation}\label{eq:uncon}
\textstyle x^*=\argmin_x~ \gamma(\|x\|_1-\|x\|_2^2) + \frac 1 2 \|Ax-b\|_2^2,
\end{equation} 
for a given positive parameter $\gamma$. 
This can be done using a similar procedure as for the $L_1$ problem~\cite{lorenz2011constructing}. In particular, any non-zero stationary point satisfies the following first-order optimality condition: 
\begin{equation}\label{eq:uncon_construct}
\textstyle \gamma\left(p^*-\frac {x^*}{\|x^*\|_2}\right)+A^\top(Ax^*-b)=0,
\end{equation}
where $p^*\in\partial\|x^*\|_1$. Denote Sign$(x)$ as the multi-valued sign, \emph{i.e.},
\begin{equation*}
y\in\mbox{Sign}(x)\ \Longleftrightarrow\ y_i\left\{\begin{array}{ll}
=1, & \mbox{if} \ x_i>0,\\
=-1, & \mbox{if} \ x_i<0,\\
\in[-1,1], & \mbox{if} \ x_i=0.
\end{array}\right.
\end{equation*}
Given $A,~\gamma$, and  $x^*$, we want to find $w\in \mbox{Sign}(x^*)$ and $w-\frac {x^*}{\|x^*\|_2} \in \mbox{Range}(A^\top)$. 
If $y$ satisfies $A^\top y=w-\frac {x^*}{\|x^*\|_2} $ and $b$ is defined by $b=\gamma y+Ax^*$, then $x^*$ is a stationary point to~\eqref{eq:uncon}. 
To find $w\in\mathbb R^N$, we consider the projection onto convex sets (POCS)~\cite{cheney1959proximity} by alternatively projecting onto two convex sets: $w\in \mbox{Sign}(x^*)$ and $w-\frac {x^*}{\|x^*\|_2} \in \mbox{Range}(A^\top)$. In particular, we compute the orthogonal basis of $A^\top$, denoted as $U$, for the sake of projecting onto the set $\mbox{Range}(A^\top)$. The iteration starts with $w^0\in\mbox{Sign}(x^*)$ and proceeds  
\begin{align*}
	 w^{k+1} = P_{\mbox{Sign}(x^*)}\left(UU^T\left(w^k - \frac {x^*}{\|x^*\|_2}\right) + {x^*\over \|x^*\|_2}\right),
\end{align*}
until  a stopping criterion is reached. The stopping condition for POCS is ether $\|w^{k+1}-w^k\|_2<1e^{-10}$ or $k>10N$. Note that  POCS may not converge and $w$ may not exist, especially when $A$ is highly coherent. 

For constructed test cases\footnote{If POCS does not converge, we discard this trial in the analysis.} with giving $A,~\gamma,~x^*,$ and $b$, we study the convergence of three $L_1$-$L_2$ implementations (DCA, FBS, and ADMM). We consider the sparse vector $x$ with sparsity 10. We fix $\lambda = 1$ (FBS stepsize), and $\delta = 0.1$ (ADMM stepsize). 
We only consider incoherent matrices (random Gaussian and partial DCT) of size $64\times 256$, as it is hard to find an optimal solution to~\eqref{eq:uncon_construct} for over-sampled DCT matrices. Figure~\ref{fig:sparsity_Gauss_DCT} shows that FBS and ADMM are much faster than the DCA in finding the stationary point $x^*$. 
Here we give a justification of the speed by complexity analysis. For each iteration, FBS requires to compute the matrix-vector multiplication of complexity $O(MN)$ and shrinkage operator of complexity $O(N)$, while ADMM requires a matrix inversion of $O(M^3)$. As for DCA, it requires to solve an $L_1$ minimization problem iteratively; at each iteration, the complexity is equivalent to FBS or ADMM, whichever we use to solve the subproblem. As a result, the DCA is much slower than FBS and ADMM.  


\begin{figure}
\centering
\begin{tabular}{cc}
(a) Gaussian, $\gamma=0.01$ & (b) DCT, $\gamma = 0.01$\\
\includegraphics[width=0.46\textwidth]{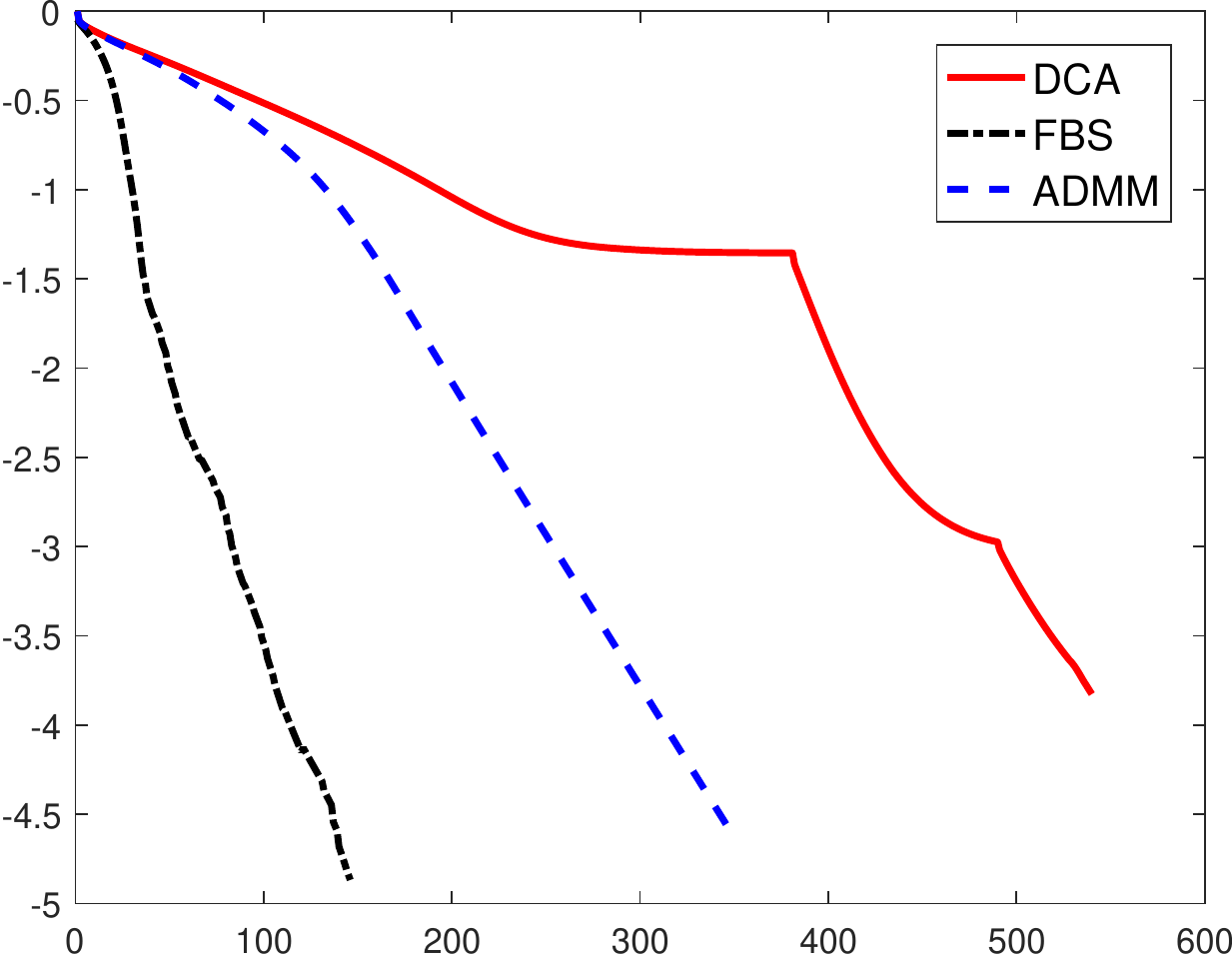}&
\includegraphics[width=0.46\textwidth]{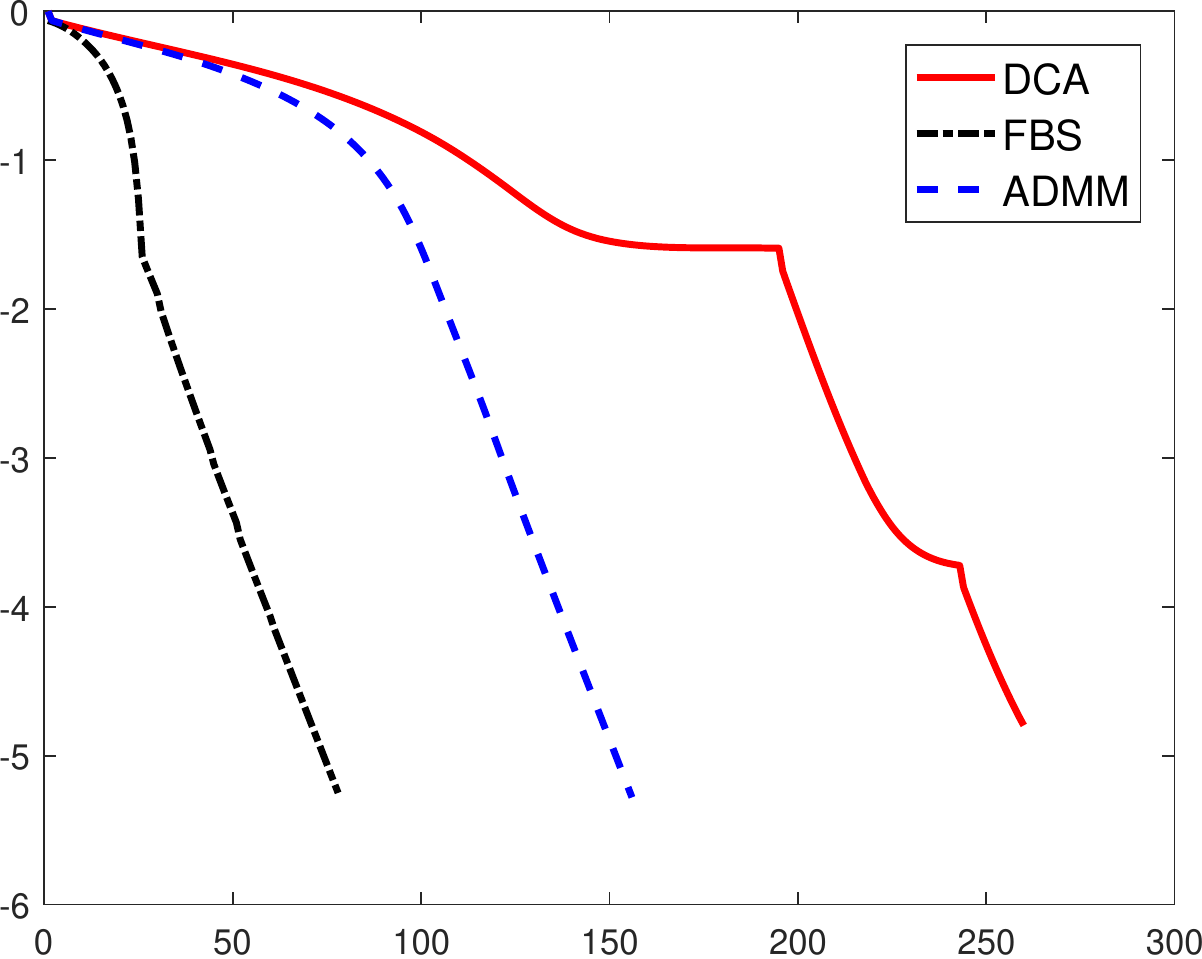}\\
(c) Gaussian, $\gamma=0.1$  & (d) Gaussian, $\gamma = 0.001$\\
\includegraphics[width=0.46\textwidth]{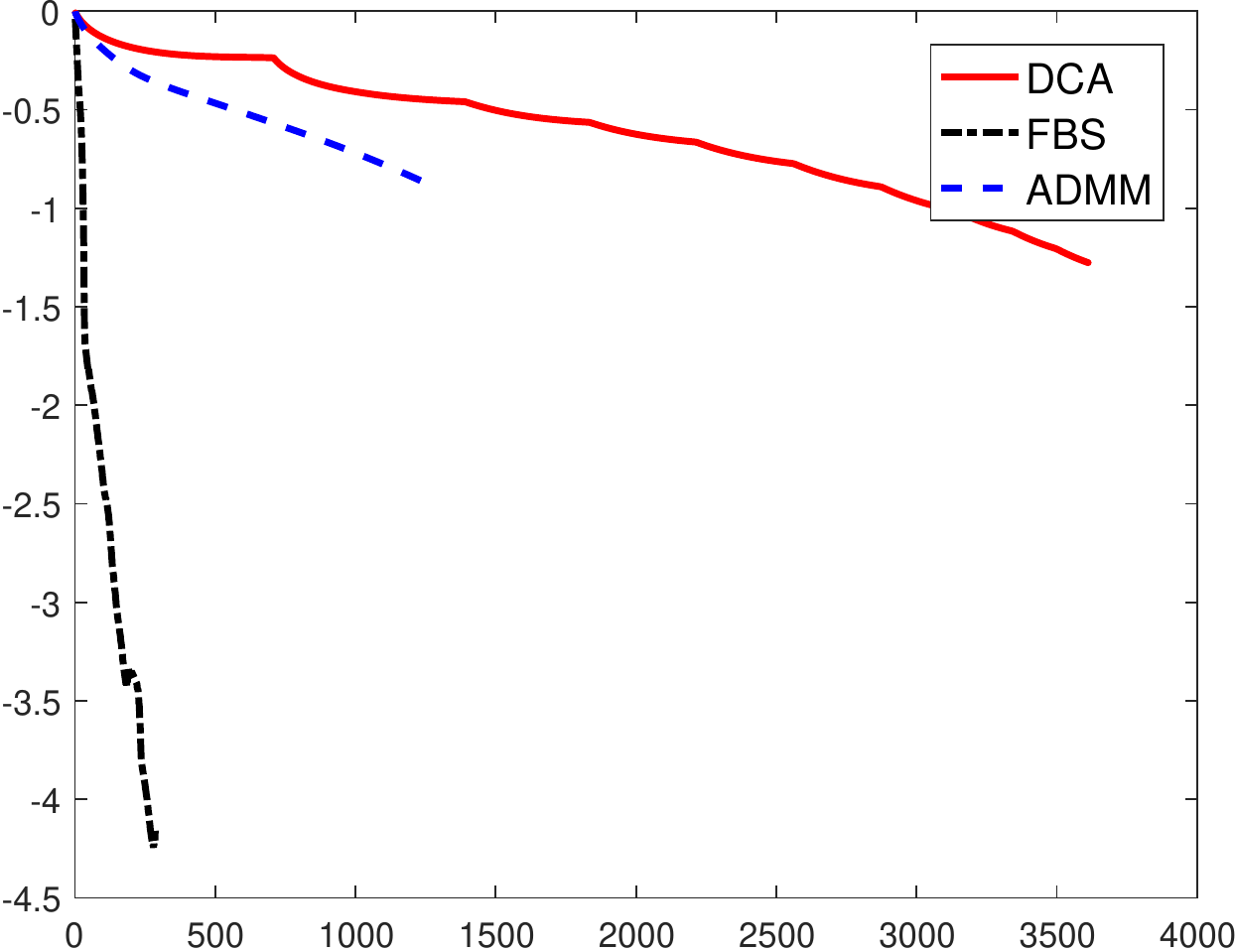}&
\includegraphics[width=0.46\textwidth]{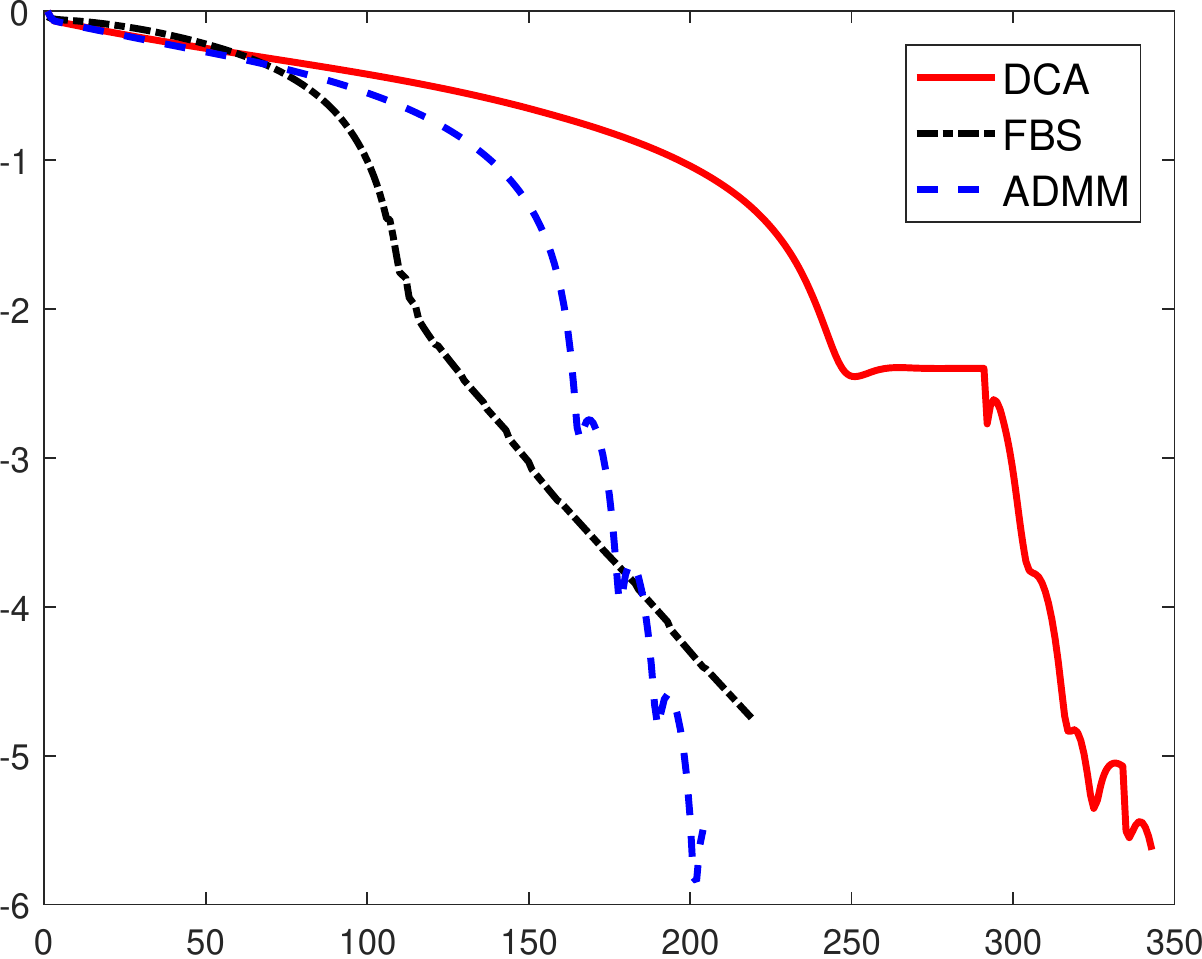}\\
\end{tabular}
\caption{Computational efficiency. Problem setting: a matrix $A$ is of size $64\times 256$ (random Gaussian or partial DCT) and $x_g$ has 10  non-zero elements drawn from standard Gaussian distribution; $b$ is constructed such that $x_g$ is a stationary point of the unconstructed $L_1$-$L_2$ minimization. In each case, we plot the error to the ground-truth solution versus iteration numbers (the number of matrix-vector multiplications divide by two because it is the most time consuming part) for three  $L_1$-$L_2$ minimization methods: DCA, FBS, and ADMM;  FBS and ADMM are much faster than DCA.  } \label{fig:sparsity_Gauss_DCT}
\end{figure}

%
%

\subsection{Noise-free case}

In this section, we look at the success rates of finding a sparse solution while satisfying the linear constraint $Ax=b$. We consider an unconstrained formulation with a small regularizing parameter in order to enforce the linear constraint.  In particular, we choose $\gamma=1e^{-6}$ for random Gaussian matrices and $\gamma=1e^{-7}$ for oversampled DCT matrices, which are shown to have good recovery results. As for algorithmic parameters, we choose  $\delta=10\gamma$ for ADMM and DCA. FBS does not work well with a very small regularization parameter  $\gamma$, while a common practice  is gradually decreasing its value. We decide not to compare with FBS in the noise-free case. Figure~\ref{fig:successrate_uncon} shows that both DCA and ADMM often yield the same solutions when sensing matrix is incoherent, \emph{e.g.}, random Gaussian and over-sampled DCT with F=5; while DCA is better than ADMM  for highly coherent matrices (bottom right plot of Figure~\ref{fig:successrate_uncon}.) We suspect the reason to be that DCA is  less prone to parameters and numerical errors than ADMM, as each DCA subproblem is convex; we will examine  extensively in the future work.
This hypothesis motivates us to design a continuation strategy of updating $\alpha$ in the weighted model of $L_1$-$\alpha L_2$. 
Particularly for incoherent matrices, we want  $\alpha$ to approach to 1 very quickly, so we 
consider  a linear update of $\alpha$ capped at 1 with a large slope. If the matrix is coherent, we want to impose a smooth transition of $\alpha$ going from zero to one, and we choose a sigmoid function to change $\alpha$ at every iteration $k$, i.e.,
\begin{equation}\label{eq:alpha}
\alpha(k) = \frac 1 {1+ae^{-rk}},
\end{equation}
where $a$ and $r$ are parameters.  We plot the evolution of $\alpha$ for over-sampled DCT when $K=5$ (incoherent) and $K=20$ (coherent) on the top right plot of Figure~\ref{fig:successrate_uncon}. Note that the iteration may stop before $\alpha$ reaches to one. We call this updating scheme a weighted model. In Figure~\ref{fig:successrate_uncon}, we show that the weighted model is  better than DCA and ADMM when the matrix is highly coherent.

\begin{figure}
\centering
\begin{tabular}{cc}
Gaussian  & The update for $\alpha$\\
\includegraphics[width=0.46\textwidth]{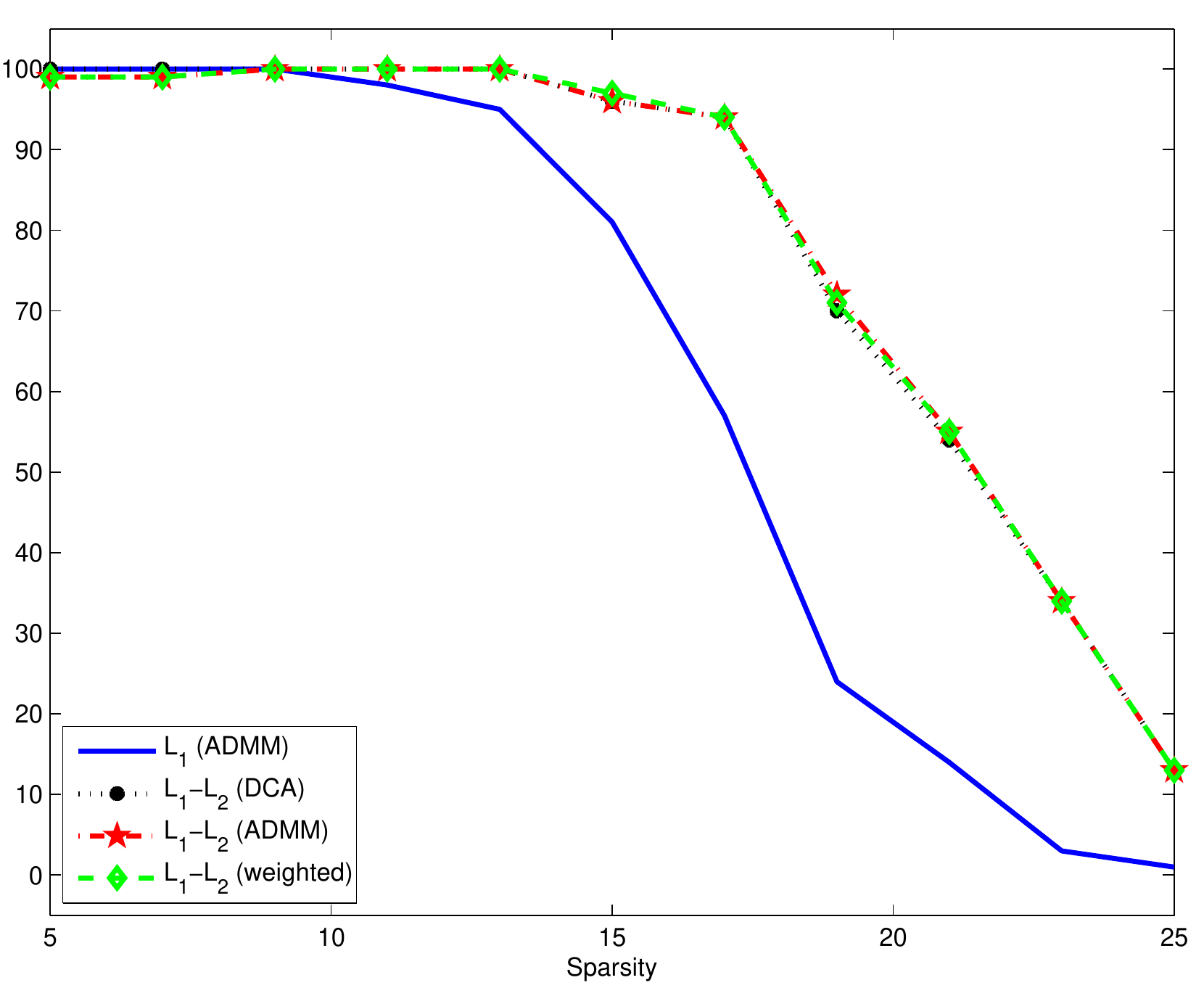}&
\includegraphics[width=0.46\textwidth]{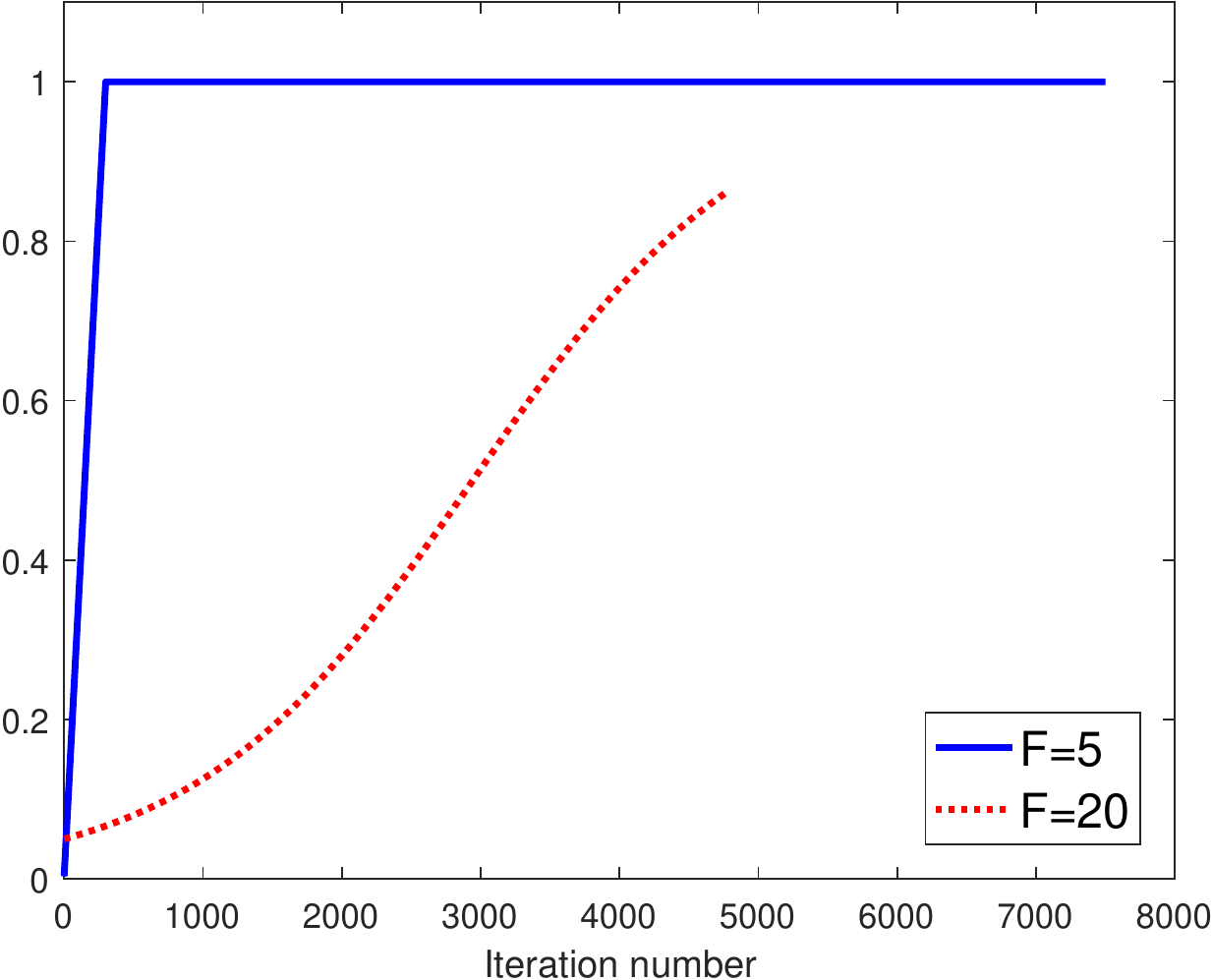}\\
$F=5$  & $F= 20$\\
\includegraphics[width=0.46\textwidth]{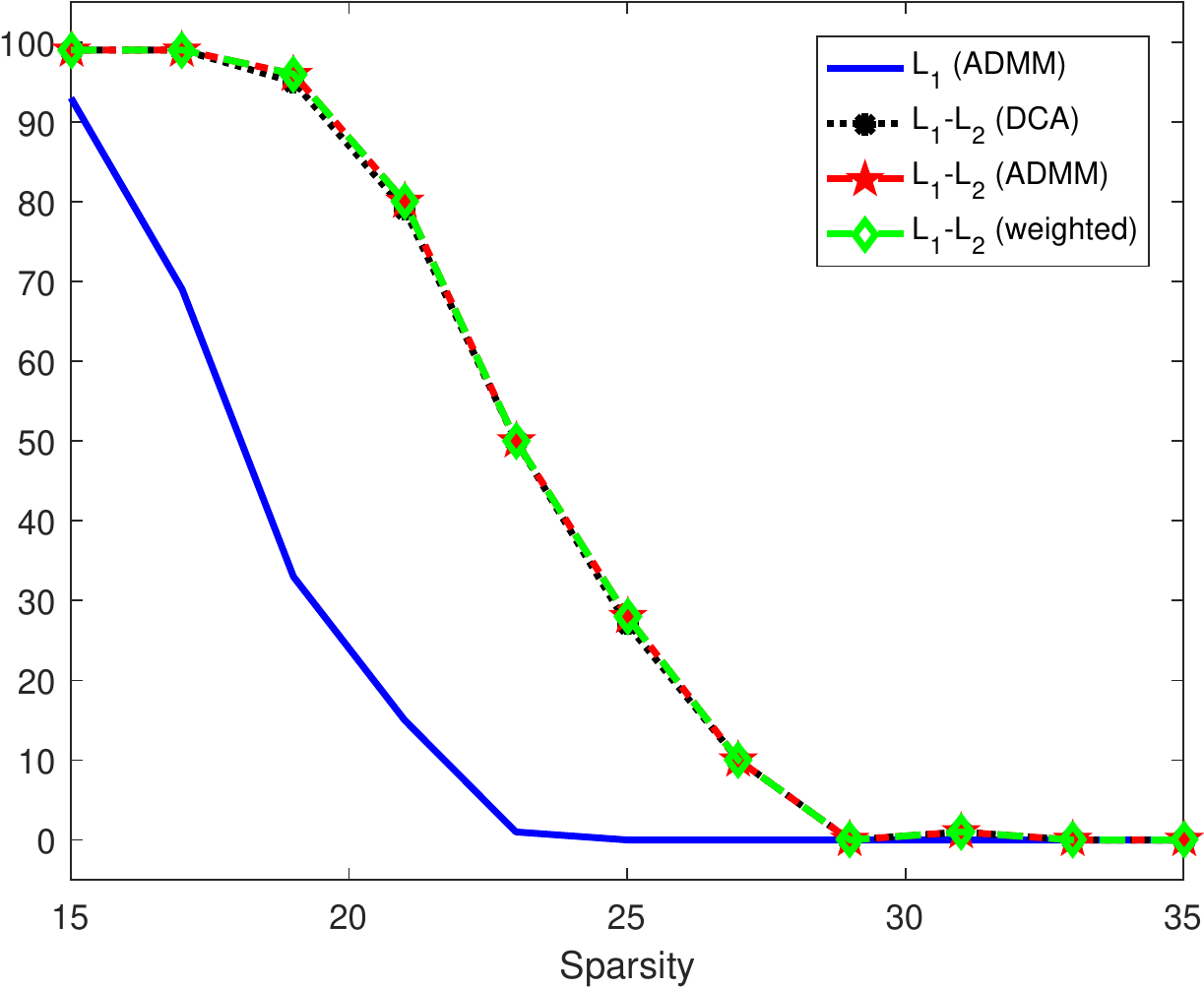}&
\includegraphics[width=0.46\textwidth]{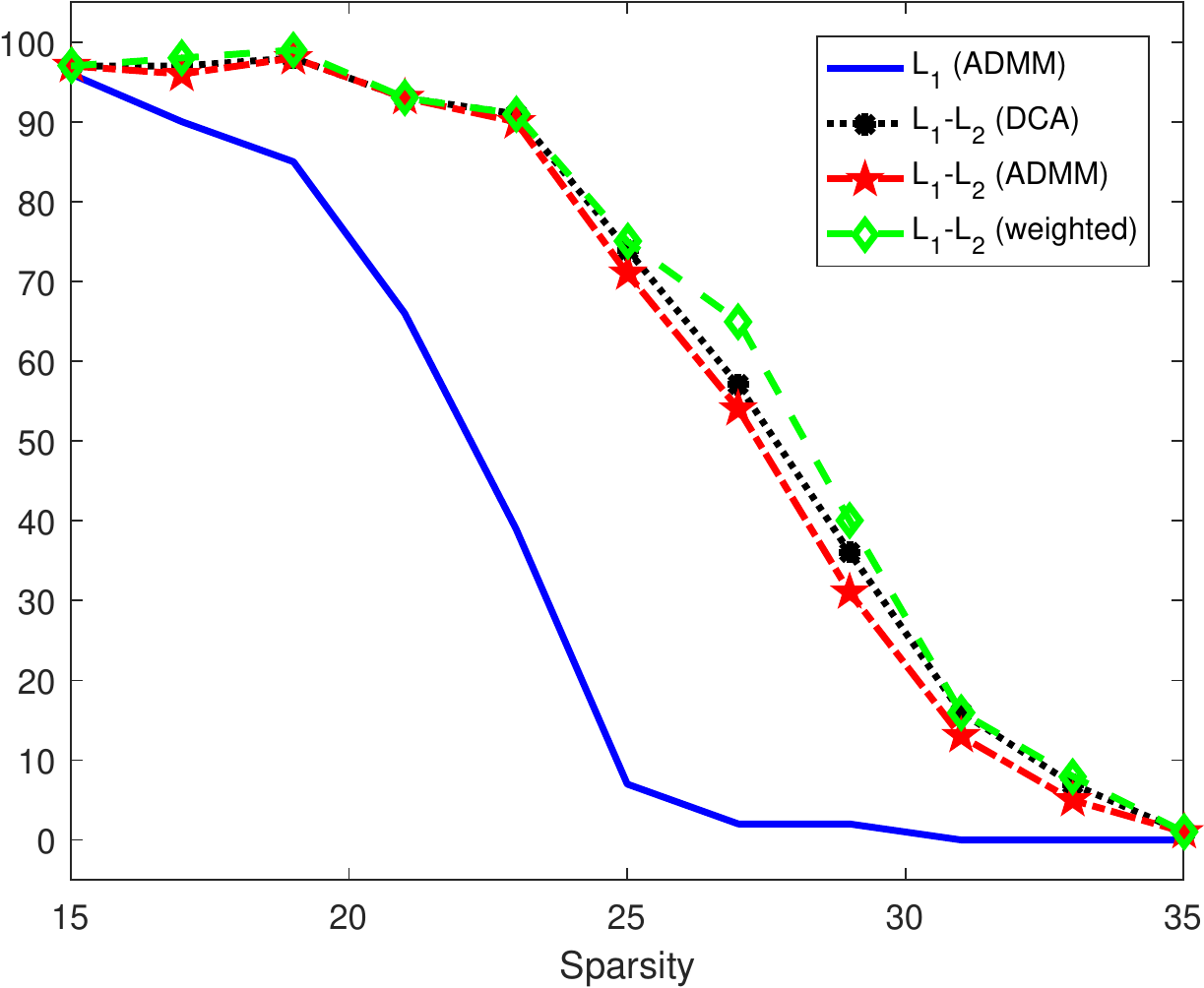}\\
\end{tabular}
\caption{Success rates of random Gaussian matrices and over-sampled DCT matrices for $F=5$ and $F=20$. The  ADMM approach yields almost the same results compared to the DCA for incoherent matrices (random Gaussian and over-sampled DCT with $F=5$), and the weighted model with a specific update of $\alpha$ (see top right plot) achieves the best results in the highly coherent case (over-sampled DCT with $F=20$). } \label{fig:successrate_uncon}
\end{figure}

Although the DCA gives better results for coherent matrices, it is much slower than ADMM  in the run time. The computational time averaged over 100 realizations for each method is reported in Table~\ref{tab:time}. DCA is almost one order of magnitude slower than ADMM and weighted model. The time for the $L_1$ minimization via ADMM is also provided. Table~\ref{tab:time} shows that $L_1$-$L_2$ via ADMM and weighted model are comparable to the $L_1$ approach in efficiency. The weighted model achieves the best recovery results in terms of both success rates and computational time.

\begin{table}
\caption{Mean and standard deviation of computational time (\textit{sec.}) for recovering $20$-sparse vectors.}
\label{tab:time}
\centering
\begin{tabular}{|l|c|c|c|c|c|}\hline
 & size & $L_1$ (ADMM) & DCA & ADMM &  weighted\\\hline
 Gaussian & $64\times 256$ & 0.06 (0.01) & 0.34 (0.14)  &  0.13 (0.02) & 0.13 (0.03)\\\hline
 DCT & $64\times 256$ & 0.06 (0.03) & 0.29 (0.15)  &  0.12 (0.02) & 0.12 (0.03)\\\hline
 F=5 & $100\times 1500$ & 0.83 (0.23) & 2.69 (1.72) &  1.09 (0.40)& 1.12 (0.40)\\\hline
 F=20 & $100\times 1500$ & 1.02 (0.04) & 3.36 (0.34) &  1.28 (0.09)  &  1.31 (0.08)\\\hline
\hline
\end{tabular}
\end{table}

\subsection{Noisy Data}

Finally we provide a series of simulations to demonstrate sparse recovery with noise, following an experimental setup in~\cite{Xu2012}. We consider a signal $x$ of length $N=512$ with $K=130$ non-zero elements. We try to recover it from $M$ measurements $b$ determined by a normal distribution matrix $A$ (then each column is normalized with zero-mean and unit norm), with white Gaussian noise of standard deviation $\sigma = 0.1$.  To compensate the noise, we use the mean-square-error (MSE) to quantify the recovery performance. 
If the support of the ground-truth solution $x$ is known, denoted as $\Lambda=\mbox{supp}(x)$, we can compute the MSE of an oracle solution, given by the formula $\sigma^2\mbox{tr}(A_{\Lambda}^TA_{\Lambda})^{-1}$, as benchmark.

We want to compare $L_1$-$L_2$ with $L_{1/2}$  via the half-thresholding method\footnote{We use the author's Matlab implementation with default parameter settings and the same stopping condition adopted as $L_1$-$L_2$ in the comparsion.}~\cite{Xu2012}, which uses an updating scheme for $\gamma$. We observe all the $L_1$-$L_2$ implementations with a fixed parameter $\gamma$ almost have the same recovery performance. In addition, we heuristically consider  to choose $\gamma$ adaptively based on the sigmoid function~\eqref{eq:alpha} with $a=-1, r = 0.02$, along with the FBS framework. Therefore, 	
	 we record  the MSE 
of two $L_1$-$L_2$ implementations: ADMM with fixed $\gamma = 0.8$ and FBS with updating $\gamma$. The $L_1$ minimization via FBS with updating $\gamma$ is also included. Each number in Figure~\ref{fig:noisyAll} is based on the average of 100 random realizations of the same setup. $L_1$-$L_2$ is better than $L_{1/2}$ when $M$ is small, but it is the other way around for large $M$. It is consistent with the observation in~\cite{yinLHX14} that $L_p$ $(0<p<1)$ is better than $L_1$-$L_2$ for incoherent sensing matrices. When $M$ is small, the sensing matrix becomes coherent, and $L_1$-$L_2$ seems to show advantages and/or robustness over $L_p$. 

In Table~\ref{tab:noisy}, we present the mean and standard deviation of MSE and computational time at four particular $M$ values: 238, 250, 276, 300, which were considered in~\cite{Xu2012}. Although the half-thresholding achieves the best results for large $M$, it is much more slower than other competing methods. We hypothesize that the convergence of $L_{1/2}$ via half-threshdoling is slower than the $L_1$-$L_2$ approach.


\begin{figure}
\centering
\includegraphics[width=0.9\textwidth]{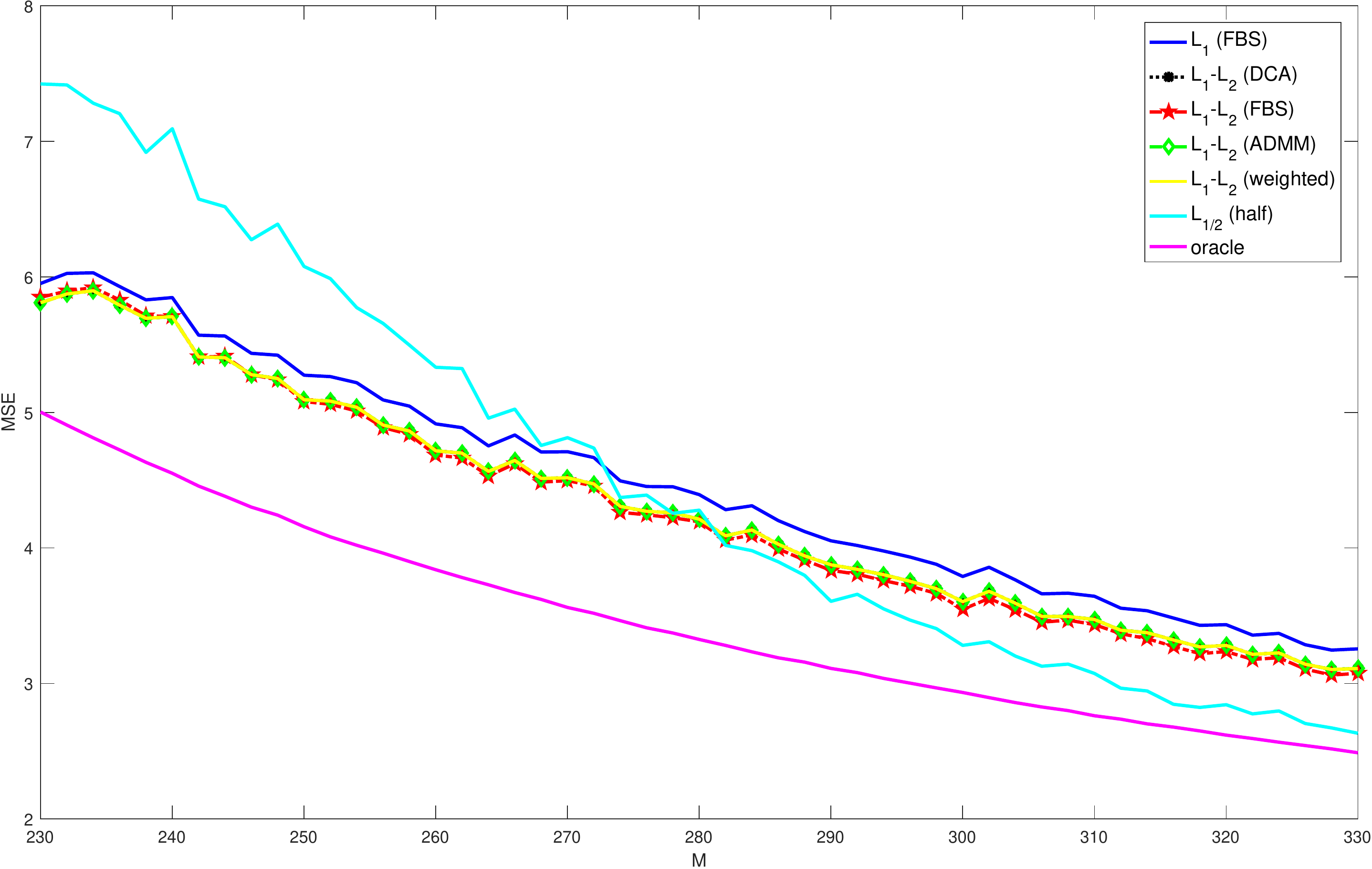}
\caption{MSE of sparse recovery under the presence of additive Gaussian white noise. The sensing matrix is of size $M\times N$, where $M$ ranges from 230 to 300 and $N = 512$. The ground-truth sparse vector contains 130 non-zero elements. The MSE values are averaged over 100 random realizations.}\label{fig:noisyAll}
\end{figure}

\begin{table}
	\centering
	\begin{tabular}{l||ccc||ccc}\hline\hline
		Methods & M & MSE & Time (sec.) & $M$ & MSE & Time (sec.)  \\\hline
		oracle  &   & 4.63 (1.00) &  &  & 4.15 (1.06) &\\
		$L_1$(FBS)  &  & 5.83 (0.74)& 0.18 (0.03)  & & 5.27 (0.65) & 0.17 (0.02)\\
		$L_1$-$L_2$(FBS) & 238& 5.71 (0.79)& 0.57 (0.26) &250 & 5.08 (0.67) & 0.49 (0.20)\\
		$L_1$-$L_2$(ADMM) && 5.69 (0.77)& 0.34 (0.09) & & 5.09 (0.65)& 0.31 (0.08) \\
		$L_{1/2}$~\cite{Xu2012} &  & 6.91 (1.00)& 1.92 (0.13) &&  6.08 (1.06) & 1.89 (0.23)\\\hline\hline
		Methods & M & MSE & Time (sec.) & $M$ & MSE & Time (sec.)  \\\hline
		oracle & & 3.41 (0.76)&& & 2.93 (0.55)&\\
		$L_1$ (FBS)  &  & 4.45 (0.51) & 0.22 (0.03)   &  & 3.79 (0.45)& 0.20 (0.02) \\
		$L_1$-$L_2$(FBS) & 276 & 4.24 (0.52)& 0.71 (0.33)&300& 3.54 (0.43)& 0.49 (0.15)\\
		$L_1$-$L_2$(ADMM) && 4.27 (0.51)& 0.25 (0.08) && 3.60 (0.43)&0.19 (0.05)\\
		$L_{1/2}$~\cite{Xu2012}  & & 4.39 (0.76)& 2.84 (0.33) & & 3.28 (0.55)& 2.99 (0.29)\\\hline
	\end{tabular}
	\caption{Recovery results of noisy signals (mean and standard deviation over 100 realizations).}\label{tab:noisy}
\end{table}

\section{Conclusions}\label{sect:conclusion}

We derived a proximal operator for $L_1$-$\alpha L_2$, as analogue to the soft shrinkage for $L_1$. This makes some fast $L_1$ solvers such as FBS and ADMM  applicable to minimize $L_1$-$\alpha L_2$. We discussed these two algorithms in details with convergence analysis. We  demonstrated numerically that FBS and ADMM together with this proximal operator are much more efficient than the DCA approach. In addition, we observed DCA gives better recovery results than ADMM for coherent matrices, which motivated us to consider a continuation strategy in terms of $\alpha$. 

\section*{Acknowledgments}
The authors would like to thank Zhi Li and the anonymous  reviewers for valuable comments.

%
%
%
%

\bibliographystyle{spmpsci}
\bibliography{l1l2}

\begin{thebibliography}{10}
\providecommand{\url}[1]{{#1}}
\providecommand{\urlprefix}{URL }
\expandafter\ifx\csname urlstyle\endcsname\relax
  \providecommand{\doi}[1]{DOI~\discretionary{}{}{}#1}\else
  \providecommand{\doi}{DOI~\discretionary{}{}{}\begingroup
  \urlstyle{rm}\Url}\fi

\bibitem{Beck2009}
Beck, A., Teboulle, M.: A fast iterative shrinkage-thresholding algorithm for
  linear inverse problems.
\newblock SIAM J. Imaging Sci. \textbf{2}(1), 183--202 (2009)

\bibitem{bredies_minimization_2015}
Bredies, K., {Lorenz}, D.A., Reiterer, S.: Minimization of non-smooth,
  non-convex functionals by iterative thresholding.
\newblock J. Optim. Theory Appl. \textbf{165}(1), 78--112 (2015)

\bibitem{candesRT06}
Cand\`es, E.J., Romberg, J., Tao, T.: Stable signal recovery from incomplete
  and inaccurate measurements.
\newblock Comm. Pure Appl. Math. \textbf{59}, 1207--1223 (2006)

\bibitem{chartrand07}
Chartrand, R.: Exact reconstruction of sparse signals via nonconvex
  minimization.
\newblock IEEE Signal Process. Lett. \textbf{10}(14), 707--710 (2007)

\bibitem{chartrandY08}
Chartrand, R., Yin, W.: Iteratively reweighted algorithms for compressive
  sensing.
\newblock In: International Conference on Acoustics, Speech, and Signal
  Processing (ICASSP), pp. 3869--3872 (2008)

\bibitem{cheney1959proximity}
Cheney, W., Goldstein, A.A.: Proximity maps for convex sets.
\newblock Proceedings of the American Mathematical Society \textbf{10}(3),
  448--450 (1959)

\bibitem{donohoE03}
Donoho, D., Elad, M.: Optimally sparse representation in general
  (nonorthogonal) dictionaries via l1 minimization.
\newblock Proc. Nat. Acad. Scien. USA \textbf{100}, 2197--2202 (2003)

\bibitem{donoho06}
Donoho, D.L.: Compressed sensing.
\newblock IEEE Trans. Inf. Theory \textbf{52}(4), 1289 -- 1306 (2006)

\bibitem{esserLX13}
Esser, E., Lou, Y., Xin, J.: A method for finding structured sparse solutions
  to non-negative least squares problems with applications.
\newblock SIAM J. Imaging Sci. \textbf{6}(4), 2010--2046 (2013)

\bibitem{fannjiangL12}
Fannjiang, A., Liao, W.: Coherence pattern-guided compressive sensing with
  unresolved grids.
\newblock SIAM J. Imaging Sci. \textbf{5}(1), 179--202 (2012)

\bibitem{gribonval2003sparse}
Gribonval, R., Nielsen, M.: Sparse representations in unions of bases.
\newblock IEEE Trans. Inf. Theory \textbf{49}(12), 3320--3325 (2003)

\bibitem{huangSY15}
Huang, X., Shi, L., Yan, M.: Nonconvex sorted l1 minimization for sparse
  approximation.
\newblock Journal of Operations Research Society of China \textbf{3}, 207--229
  (2015)

\bibitem{krishnanF09}
Krishnan, D., Fergus, R.: Fast image deconvolution using hyper-{L}aplacian
  priors.
\newblock In: Advances in Neural Information Processing Systems (NIPS), pp.
  1033--1041 (2009)

\bibitem{laiXY13}
Lai, M.J., Xu, Y., Yin, W.: Improved iteratively reweighted least squares for
  unconstrained smoothed lq minimization.
\newblock SIAM J. Numer. Anal. \textbf{5}(2), 927--957 (2013)

\bibitem{li2015global}
Li, G., Pong, T.K.: Global convergence of splitting methods for nonconvex
  composite optimization.
\newblock SIAM J. Optim. \textbf{25}, 2434--2460 (2015)

\bibitem{li2015accelerated}
Li, H., Lin, Z.: Accelerated proximal gradient methods for nonconvex
  programming.
\newblock In: Advances in Neural Information Processing Systems, pp. 379--387
  (2015)

\bibitem{liu2016further}
Liu, T., Pong, T.K.: Further properties of the forward-backward envelope with
  applications to difference-of-convex programming.
\newblock Computational Optimization and Applications  (2017)

\bibitem{lorenz2011constructing}
Lorenz, D.A.: Constructing test instances for basis pursuit denoising.
\newblock Trans. Sig. Proc. \textbf{61}(5), 1210--1214 (2013)

\bibitem{louOX15}
Lou, Y., Osher, S., Xin, J.: Computational aspects of l1-l2 minimization for
  compressive sensing.
\newblock In: Model. Comput. \& Optim. in Inf. Syst. \& Manage. Sci., Advances
  in Intelligent Systems and Computing, vol. 359, pp. 169--180 (2015)

\bibitem{louYHX14}
Lou, Y., Yin, P., He, Q., Xin, J.: Computing sparse representation in a highly
  coherent dictionary based on difference of l1 and l2.
\newblock J. Sci. Comput. \textbf{64}(1), 178--196 (2015)

\bibitem{louYX15}
Lou, Y., Yin, P., Xin, J.: Point source super-resolution via non-convex l1
  based methods.
\newblock Journal of Scientific Computing \textbf{68}(3), 1082--1100 (2016)

\bibitem{mammone83}
Mammone, R.J.: Spectral extrapolation of constrained signals.
\newblock J. Opt. Soc. Am. \textbf{73}(11), 1476--1480 (1983)

\bibitem{natarajan95}
Natarajan, B.K.: Sparse approximate solutions to linear systems.
\newblock SIAM J. comput. \textbf{24}, 227--234 (1995)

\bibitem{papoulisC79}
Papoulis, A., Chamzas, C.: Improvement of range resolution by spectral
  extrapolation.
\newblock Ultrasonic Imaging \textbf{1}(2), 121--135 (1979)

\bibitem{TA98}
Pham-Dinh, T., Le-Thi, H.A.: A {DC} optimization algorithm for solving the
  trust-region subproblem.
\newblock SIAM J. Optim. \textbf{8}(2), 476--505 (1998)

\bibitem{repetti2015euclid}
Repetti, A., Pham, M.Q., Duval, L., Chouzenoux, E., Pesquet, J.C.: Euclid in a
  taxicab: Sparse blind deconvolution with smoothed regularization.
\newblock IEEE Signal Processing Letters \textbf{22}(5), 539--543 (2015)

\bibitem{rockafellar2015convex}
Rockafellar, R.T.: Convex analysis.
\newblock Princeton university press (1997)

\bibitem{rockafellar_variational_2009}
{Rockafellar}, R.T., {Wets}, R.J.B.: Variational analysis.
\newblock {Springer}, Dordrecht (2009)

\bibitem{santosaS86}
Santosa, F., Symes, W.W.: Linear inversion of band-limited reflection
  seismograms.
\newblock SIAM J. Sci. Stat. Comp. \textbf{7}(4), 1307--1330 (1986)

\bibitem{wang_global_2015}
{Wang}, Y., {Yin}, W., {Zeng}, J.: Global convergence of {ADMM} in nonconvex
  nonsmooth optimization.
\newblock arXiv:1511.06324 {[}cs, math]  (2015)

\bibitem{woodworth2015compressed}
Woodworth, J., Chartrand, R.: Compressed sensing recovery via nonconvex
  shrinkage penalties.
\newblock Inverse Problems \textbf{32}(7), 075,004 (2016)

\bibitem{wuSL15}
Wu, L., Sun, Z., Li, D.H.: A {B}arzilai--{B}orwein-like iterative half
  thresholding algorithm for the $l_{1/2}$ regularized problem.
\newblock J. Sci. Comput. \textbf{67}, 581--601 (2016)

\bibitem{Xu2012}
Xu, Z., Chang, X., Xu, F., Zhang, H.: $l_{1/2}$ regularization: A thresholding
  representation theory and a fast solver.
\newblock IEEE Trans. Neural Netw. Learn. Syst. \textbf{23}, 1013--1027 (2012)

\bibitem{yinLHX14}
Yin, P., Lou, Y., He, Q., Xin, J.: Minimization of $l_1 - l_2$ for compressed
  sensing.
\newblock SIAM J. Sci. Comput. \textbf{37}, A536--A563 (2015)

\bibitem{zhangX14}
Zhang, S., Xin, J.: Minimization of transformed $l_1$ penalty: Theory,
  difference of convex function algorithm, and robust application in compressed
  sensing.
\newblock arXiv preprint arXiv:1411.5735  (2014)

\end{thebibliography}

\end{document}